\documentclass[a4paper,oneside,10pt]{amsart}

\usepackage{amsaddr}

\usepackage[english]{babel}
\usepackage[utf8]{inputenc}
\usepackage{svg}
\usepackage[a4paper,top=2cm,bottom=2cm,left=2cm,right=2cm,marginparwidth=1.75cm]{geometry}
\usepackage{mathtools,amssymb,amsthm,bbold,mathrsfs,enumerate}
\usepackage{tikz-cd, tikz}
\usepackage{graphicx}
\usepackage[colorlinks=true, allcolors=blue]{hyperref}
\usepackage{adjustbox}
\usepackage{tensor}
\usepackage{mathrsfs}
\usepackage{xcolor}
\usepackage[backend=biber, maxnames=50, style=numeric]{biblatex}
\usepackage{csquotes}
\usepackage{setspace}
\allowdisplaybreaks
\tikzset{
    labl/.style={anchor=south, rotate=90, inner sep=.5mm}
}
\theoremstyle{plain}
\newtheorem{prop}{Proposition}[section]
\newtheorem{lem}[prop]{Lemma}
\newtheorem{thm}[prop]{Theorem}

\theoremstyle{definition}

\newtheorem{defn}[prop]{Definition}
\newtheorem{rem}[prop]{Remark}

\newcommand*\circled[1]{\tikz[baseline=(char.base)]{
            \node[shape=circle,draw,inner sep=2pt] (char) {#1};}}
\newcommand{\Sym}{\operatorname{Sym}}
\newcommand{\extbr}{\widetilde{b}}
\newcommand{\extcobr}{\widetilde{c}}
\newcommand{\br}{b}
\newcommand{\cobr}{c}
\newcommand{\id}{\textnormal{id}}
\newcommand{\str}{\operatorname{str}}
\newcommand{\sdim}{\operatorname{sdim}}

\title{Graded Necklace Lie Bialgebras and Batalin--Vilkovisky Formalism}
\author{Nikolai Perry}
\address{Section of Mathematics, University of Geneva, Geneva, Switzerland}
\email{nikolaihperry@gmail.com}
\author{Ján Pulmann}
\address{Mathematical Institute, Faculty of Mathematics and Physics, Charles University,
Prague, Czech Republic}
\email{jan.pulmann@gmail.com}

\begin{document}

\begin{abstract}
   An involutive Lie bialgebra induces a Batalin-Vilkovisky operator on its exterior algebra. We introduce a graded generalization of the necklace Lie bialgebra, which depends on a choice of a quiver $Q$. We relate the resulting Batalin-Vilkovisky structure to the Batalin-Vilkovisky structure coming from a degree $-1$ symplectic form on a suitably defined representation variety of the quiver $Q$. The morphism intertwining these Batalin-Vilkovisky algebras will be given by a twisted trace, recovering the usual (super)trace and the odd trace.
\end{abstract}

\maketitle

\section{Introduction}

The necklace Lie algebra $\mathscr A$ is a Lie algebra constructed from a quiver $Q$ --- it is given by the free  vector space on the set of all cyclic paths in $\overline{Q}$ (the double of $Q$), together with a Lie bracket that glues these paths together at all pairs of arrows in involution. This Lie algebra is related to the representation variety of $\overline{Q}$ --- which carries a natural symplectic structure --- via a trace map, which intertwines the Lie bracket and the canonical Poisson bracket \cite{BLB, Gin}. 

A Lie cobracket compatible with the original necklace Lie bracket was then introduced by Schedler \cite{Schedler}, turning $\mathscr A$ into an involutive Lie bialgebra. Already in \cite{GanSchedler}, the problem of interpretation of the cobracket was raised, and an answer in terms of Connes-Kreimer renormalization Hopf algebras was given. In this paper, we relate the necklace Lie bialgebra to an odd symplectic structure on the representation variety. Namely, following a similar construction for surfaces \cite{ANPS}, we start with the following simple observation. The necklace Lie bialgebra, being involutive, can be encoded in a Batalin-Vilkovisky (BV) operator on the super-commutative algebra $\bigwedge \mathscr A$. On the other hand, we consider a subspace of the representation variety of the doubled quiver $\overline{Q}$ given by linear maps intertwining fixed odd endomorphisms associated to vertices. This subspace carries an odd non-degenerate pairing, which induces a canonical BV operator on $\widetilde{\mathcal{O}}$, the algebra of functions on the intertwining representation variety. Our first result is a proof that a natural trace map $\bigwedge \mathscr A \to \widetilde{\mathcal O}$ is a morphism of BV algebras.

In addition, inspired by Barannikov \cite{BarannikovModularOperads, BarannikovNoncommutativeBatalinVilkoviskyGeometry, BarannikovMatrixDeRhamcomplex, BarannikovSolving}, we introduce a parameter $p \in \mathbb Z_2$ to the above construction. The case $p=0$ relates the usual Schedler necklace Lie bialgebra with a quiver version of the representation variety for the queer Lie superalgebra; this is the quiver analogue of \cite[Sec.~4]{BarannikovMatrixDeRhamcomplex} and \cite{ANPS}. In the case of $p=1$, the necklace Lie bialgebra operations have degree $-1$, while the representation variety more closely resembles the usual representation variety of the quiver. Barannikov's BV algebras constructed for even inner products ($p=0$) and odd inner products ($p=1$) correspond to a one-vertex quiver in the present work. 

\subsection{Relationship with other works}
Let us now elaborate more on how the present works fits with existing literature. As we mentioned above, there is a parallel story, with quivers replaced by surfaces. The relevant Lie bialgebra was described by Goldman \cite{GoldmanInvariant}, Turaev \cite{Turaev}, and Chas \cite{Chas2004}; Goldman related his bracket to the Atiyah-Bott structure on the character variety of the surface \cite{AtiyahBott, GoldmanInvariant}, while a similar construction involving the cobracket was recently given by \cite{ANPS}, using Batalin-Vilkovisky algebras and the character variety for the queer Lie supergroup. 

One can understand quivers in the present work as infinitesimal versions of surfaces, and quiver varieties as infinitesimal versions of character varieties of surfaces. These are related by so-called expansions \cite[Sec.~2.2]{BarNatanDancso}, which can be understood as universal formal Poisson morphisms from the quiver variety to the representation variety \cite[Sec.~1.2]{KawazumiKunoLogarithms} \cite{MassuyeauTuraevFox, Naef}; see also \cite[Thm.~6.17]{AKKN} and references therein for further developments. The relevant quivers are star-shaped in this case, and it is natural to expect that more general quivers are related to marked quilted surfaces of \cite{LBSQuilted}. 
\medskip

The analogue of the work of Goldman and Turaev in the world of quivers is due to Bocklandt--Le Bruyn, Ginzburg and Schedler \cite{BLB, Gin, Schedler}. For a quiver with a single vertex, Schedler's Lie bialgebra was related to the canonical Batalin-Vilkovisky structure on the representation variety by Barannikov in \cite[Sec.~4]{BarannikovMatrixDeRhamcomplex}. Let us mention in passing that this Lie bialgebra is a special case of Barannikov's construction in \cite{BarannikovModularOperads} obtained for a choice of the associative modular operad. Furthermore, Barannikov also considered the possibility of  assigning an \emph{odd} degree to the new edges of the doubled (one-vertex) quiver. 
\medskip

With this context in mind, we can now characterize our contribution in two ways: 
\begin{itemize}
    \item The present work gives a representation--theoretic interpretation of the whole necklace Lie bialgebra \cite{BLB, Gin, Schedler}, analogous to how the work \cite{ANPS} gives a flat--connections--theoretic interpretation of the Goldman--Turaev Lie bialgebra \cite{GoldmanInvariant, Turaev}; and
    \item The present work is the generalization to arbitrary quivers of the work of Barannikov \cite{BarannikovModularOperads, BarannikovNoncommutativeBatalinVilkoviskyGeometry, BarannikovMatrixDeRhamcomplex, BarannikovSolving}.
\end{itemize}
In fact, to orient oneself in the literature above, it might be helpful to explicitly state the possible binary ``axes'' on which it is possible to move:
\smallskip

\begin{center}
    \centering
    \begin{tabular}{lcl}
         quivers &$\leftrightarrow$ & surfaces \\
         no representation variety &$\leftrightarrow$ &geometry of representation variety \\
         
         Lie algebras, Poisson brackets &$\leftrightarrow$ &Lie bialgebras, BV operators \\
         $p=0$ &$\leftrightarrow$& arbitrary $p$ \\
         single vertex quivers &$\leftrightarrow$& arbitrary quivers \\
         associative operad &$\leftrightarrow$& arbitrary operads  \\
         classical &$\leftrightarrow$& quantum \\
    \end{tabular}
\end{center}
\smallskip

Our work corresponds to columns LRRRRLL. The above table suggests possible future directions of research, such as combinations of quivers and operads, surfaces with an odd intersection pairing or various quantizations. 
\medskip

Finally, let us mention that Lie bialgebras with non-zero degrees of the bracket and the cobracket appeared in multiple different contexts \cite{MWLieBialgProp, KMW2016, CFL}; we compare them to our version in Appendix \ref{Appendix:IBL}.

\subsection{Conventions}
The ground field is fixed to be $\mathbb{C}$, for convenience. All diagrams are to be read bottom-to-top and understood in the symmetric monoidal category of $\mathbb Z_2$-graded vector spaces (i.e. with degree-preserving maps as morphisms).

\begin{rem}
    Our original goal was to allow for $\mathbb Z$-graded vector spaces, but for reasons explained in Appendix \ref{Appendix:ZZ2} we decided to use $\mathbb Z_2$-graded vector spaces throughout. However, we still specify the would-be $\mathbb Z$-grading wherever possible, i.e. our BV operator $\Delta$ is said to be of degree $1$, which means it is odd. See Appendix \ref{Appendix:ZZ2} for an explanation on which parts generalize to the  $\mathbb Z$-graded setting. 
\end{rem}

\subsection{Organization of the paper}
In \S\ref{section: GLBs and BV Algebras} we remind the reader of BV algebras and their relation to involutive Lie bialgebras. Then, in Definition \ref{defn: GLBv2} we introduce Lie bialgebras where the bracket and the cobracket have identical but possibly non-zero degree. We give our $\mathbb{Z}_2$-graded generalisation of Schedler's necklace Lie bialgebra in \S\ref{section: GNLBs} and introduce the corresponding BV algebra. In \S\ref{section: BV Algebra on a Representation Variety} we define a BV algebra on a representation variety of the relevant quiver in coordinate-free language, after which we give its coordinate description which is useful for later proofs. In \S\ref{section: A BV algebra 
morphism}, we construct a BV morphism between these BV algebras and in \S\ref{section: examples}, we describe our results for the quiver with one edge and two vertices.

Finally we comment on the issue of $\mathbb Z$ and $\mathbb Z_2$ gradings in Appendix \ref{Appendix:ZZ2}, review existing notions of graded Lie bialgebras in Appendix \ref{Appendix:IBL} and sketch a graphical proof of the Lie bialgebra axioms for the necklace Lie bialgebra in Appendix \ref{Appendix:Cals}. 

\subsection{Acknowledgements} We would like to thank Anton Alekseev for his interest in our work, and numerous useful discussions prior to and during the writing this paper. J.P. would also like to thank Florian Naef for discussions, as well as Sebastian Schlegel Mejia, Shivang Jindal and Sarunas Kaubrys for answering his questions about quiver varieties. 

Research of J.P. was supported by the Postdoc.Mobility grant 203065 of the SNSF. J.P. would also like to thank Humboldt University of Berlin for hospitality. Research of N.P. was supported in part by grant number 208235 and by the National Center for Competence in Research (NCCR) SwissMAP of the SNSF, as well as by the University of Edinburgh's School of Mathematics research scholarship.

\section{Lie bialgebras and Batalin-Vilkovisky algebras}
\label{section: GLBs and BV Algebras}
  
    In this section, we will recall the well-known construction of Batalin-Vilkovisky algebras from involutive Lie bialgebras \cite[Sec.~5]{CMW} or \cite[Sec.~2]{CFL}. We will describe a variant for $\mathbb Z_2$-graded Lie bialgebras, with the bracket and the cobracket of the same but not necessarily even parity. To keep the sign issues manageable, we will dutifully use diagrammatic calculus, i.e. string diagrams in the symmetric monoidal category of supervector spaces (see e.g. \cite{Selinger2010} for graphical language for symmetric monoidal categories). 
    
     
    \begin{defn}[BV Algebra]\label{defn: BV algebra}
        A \emph{$\mathbb Z_2$-graded Batalin--Vilkovisky (BV) algebra} is a $\mathbb Z_2$-graded unital supercommutative algebra $\mathscr{B}$ together with a degree 1 second-order differential operator $\Delta \colon \mathscr{B} \to \mathscr{B}$ satisfying $\Delta^2 = 0$ and $\Delta(1) = 0$.
    \end{defn}
    \begin{rem} \label{rem:7T}
        In Definition \ref{defn: BV algebra}, the condition that $\Delta$ is a second-order differential operator is given by the well-known \textit{seven term identity}
        \begin{align*}
            \Delta(xyz) &- \Delta(xy)z - (-1)^{x(y+z)} \Delta(yz)x - (-1)^{z(x+y)}\Delta(zx)y\\
            &+ \Delta(x)yz + (-1)^{x(y+z)}\Delta(y)zx + (-1)^{z(x+y)}\Delta(z)xy = 0,
        \end{align*}
        for all (homogeneous) $x,y,z \in \mathscr{B}$. Note that letting $x,y,z=1$ in this condition implies $\Delta(1)=0$. For a general differential operator with a possibly non-zero constant term, the correct notion of being second order has an additional term $\Delta(1)xyz$ on the RHS.
    \end{rem}
    Let us now assume that $\mathscr{B}$ is equal to $\Sym {V}$ as a supercommutative algebra, for some $\mathbb Z_2$-graded vector space $V$. Moreover, we will only work with BV operators $\Delta$ constructed as follows. Let $\br \colon V\otimes  V \to  V$ and $\cobr \colon V \to  V\otimes  V$ be degree $1$ linear maps satisfying the following symmetry properties:
    \begin{equation}\label{eqn: symmetry properties}
        \vcenter{\hbox{\includegraphics[scale=1.2]{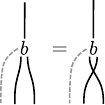}}}\;\;,\quad
        \vcenter{\hbox{\includegraphics[scale=1.2]{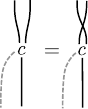}}}\;.
    \end{equation}
    The diagrams are to be read from bottom to top. The dashed gray line denotes the 1-dimensional vector space $S$ concentrated in degree 1, i.e. instead of a degree $1$ map $\br \colon V \otimes V \to V$ we consider a degree 0 map $S\otimes V \otimes V \to V$. The crossing denotes the symmetry morphism $V\otimes V \to V \otimes V$ with the usual Koszul sign.

    Thanks to these symmetry properties, we can define degree $1$ maps  $\extbr\colon \Sym^2 V \to \Sym^1 V$ and $\extcobr:\Sym^1 V \to \Sym^2 V$. They are related to $\br$ and $\cobr$ by
    \begin{equation*}
        \vcenter{\hbox{\includegraphics[scale=1.2]{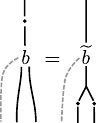}}}\;\;,\quad
        \vcenter{\hbox{\includegraphics[scale=1.2]{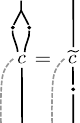}}}\;,
    \end{equation*}
    where the dots denote the inclusion $V \hookrightarrow \Sym  V$. 
    
    We then extend $\extbr$ to a degree 1 second order differential operator\footnote{This is equivalent to requiring that $\extbr$ satisfies the seven-term identity from Remark \ref{rem:7T}.} on $\Sym V$. That is, on a symmetric product $x_1\cdots x_n$ with $n \geq 2$ (where each $x_i \in \Sym^{1} V = V$) we have
    \begin{equation*}
        \extbr(x_1\cdots x_n) = \sum_{i < j} (-1)^{\varepsilon} \extbr(x_i x_j) \, x_1 \cdots \hat{x}_i \cdots \hat{x}_j \cdots x_n,
    \end{equation*}
    where $\varepsilon$ denotes the appropriate Koszul sign and the hats denote omission, and $\extbr$ vanishes on $\Sym^{\leq 1}V$.

    Similarly, we extend $\extcobr$ to a degree 1 derivation on $\Sym V$. That is, on a symmetric product $x_1 \cdots x_n$ with $n \geq 1$ we have
    \begin{equation*}
        \extcobr(x_1\cdots x_n) = \sum_i (-1)^{\varepsilon}\extcobr(x_i)x_1\cdots\hat{x}_i\cdots x_n, 
    \end{equation*}
    and $\extcobr$ vanishes on $\Sym^0 V \cong \mathbb C$.

    Let us now give the conditions which make $\extbr + \extcobr$ into a BV operator.
    \begin{prop}\label{prop:constructing BVA} Let $\br \colon V \otimes V \to V$ and $\cobr \colon V \to V \otimes V$ be degree 1 maps satisfying \eqref{eqn: symmetry properties}. Then $\Delta = \extbr + \extcobr$ defines a structure of a BV algebra on $\Sym V$ if and only if the following four equalities involving $\br$ and $\cobr$ hold:
        \begin{equation}\label{eqn: GLB conditions}
            \vcenter{\hbox{{\includegraphics[scale=1.2]{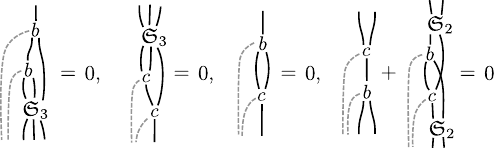}}}}\;\;.
        \end{equation}
    \end{prop}
    Here, $\mathfrak{S}_n$ are the cyclic symmetrisers
    \begin{align*}
         \mathfrak{S}_n \colon V^{\otimes n} &\to  V^{\otimes n}; \\
         \quad x_1 \otimes \cdots \otimes x_n &\mapsto x_1 \otimes \cdots \otimes x_n \pm x_2 \otimes \cdots \otimes x_1 \pm \dots \pm  x_n \otimes \cdots \otimes x_{n-1},
    \end{align*} with $\pm$ given by the Koszul sign of the cyclic permutation.
    
    The equalities in \eqref{eqn: GLB conditions} are shifted versions the \textit{Jacobi identity}, the \textit{co-Jacobi identity}, \textit{involutivity}, and the \textit{cocycle condition}, respectively. In standard notation, these can be written as
    \begin{align}
        &\br(\br \otimes 1) \mathfrak{S}_3 = 0, \quad
        \mathfrak{S}_3(\cobr \otimes 1) \cobr = 0, \quad 
        \br \cobr = 0, \nonumber \\
        &\cobr(\br(x,y)) = (-1)^{x+1}\mathrm{ad}_x (\cobr(y)) + (-1)^{(x+1)y + 1}\mathrm{ad}_y (\cobr(x)), \label{eq:cocycle}
    \end{align}
    for arbitrary homogeneous $x,y \in V$, where $\mathrm{ad}_x$ acts on $V \otimes V$ as $\mathrm{ad}_x \otimes 1 + 1 \otimes \mathrm{ad}_x$ (and on $V$, $\mathrm{ad}_x$ is simply $\br(x,-)$).
 
    \begin{proof}
Let $I_n$ denote the natural injective map 
    \begin{equation*}
        \Sym^{n} V \hookrightarrow  V^{\otimes n};\quad
        x_1 \cdots x_n \mapsto \sum_{\sigma \in S_n}(-1)^{\varepsilon}x_{\sigma(1)} \otimes \cdots \otimes x_{\sigma(n)},
    \end{equation*} where $\varepsilon$ denote the relevant Koszul signs. To prove that we get a BV algebra structure on $\Sym V$, we just need to show that $\Delta^2 = 0$. Since $\Delta$ is odd, we have that $\Delta^2 = \frac{1}{2}[\Delta,\Delta]$ is third order. Thus, $\Delta^2$ is identically zero precisely when it vanishes on the subspace $\Sym^{\leq 3} V \subset \Sym V$. It is immediate from the definition that $\Delta^2$ is zero on $\Sym^{0} V \cong  \mathbb{C}$. We thus consider $\Sym^{i} V$ for $i=1,2,3$.

    \begin{itemize}
        \item On $\Sym^{1} V \cong V$:\\
        For any $x \in V$, we have $\Delta^2(x) = \extbr(\extcobr(x)) + \extcobr(\extcobr(x))$. Since $\extbr(\extcobr(x))$ and $\extcobr(\extcobr(x))$ have distinct polynomial degrees, it follows that $\Delta^2$ vanishes on $\Sym^{1} V$ precisely when both $\extbr\, \extcobr$ and $\extcobr\,^2$ do. For the $\extbr\, \extcobr$ term we have
        \begin{equation*}
            \vcenter{\hbox{\includegraphics[scale = 1.3]{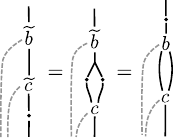}}}\;\;,
        \end{equation*}
        and so, composing with $I_1$ (which is injective) and using that
        \begin{equation*}
            \vcenter{\hbox{\includegraphics[scale = 1.3]{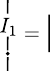}}}\;\;,
        \end{equation*}
        we see that $\extbr\,\extcobr$ is zero on $\Sym^{1} V$ if and only if $\br\cobr = 0$ --- involutivity.
        As for $\extcobr\,^2$, we have
        \begin{equation*}
            \vcenter{\hbox{\includegraphics[scale=1.2]{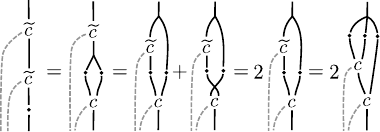}}}.
        \end{equation*}
        Composing this term with $I_3$ gives
        \begin{equation*}
            \vcenter{\hbox{\includegraphics[scale=1.2]{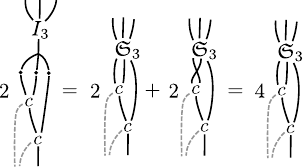}}}\;\;,
        \end{equation*}
        from which it follows that $\extcobr\,^2$ is zero on $\Sym^{1} V$ if and only if $\mathfrak{S}_3(\cobr \otimes 1)\cobr$ = 0 --- co-Jacobi. Notice that, since $\extcobr\,^2 = \frac{1}{2}[\,\extcobr,\extcobr\,]$ has order 1 (and $\extcobr$ already vanishes on $\Sym ^0 V$), this condition is equivalent to that of $\extcobr\,^2$ vanishing on $\Sym V$.

        \item On $\Sym^{2} V$:\\
        Here we assume that $\Delta^2$ vanishes on $\Sym^{\leq1}V$. Then $\extcobr\,^2 = 0$ as above, so for any $x,y \in V$ we have $\Delta^2(xy) = \extbr(\extcobr(xy)) + \extcobr(\extbr(xy)) = [\,\extbr,\extcobr \,](xy)$. For the $\extcobr\, \extbr$ term we have
        \begin{equation*}
            \vcenter{\hbox{\includegraphics[scale=1.2]{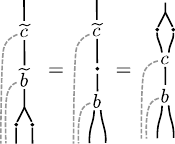}}}\;\;,
        \end{equation*}
        while for the $\extbr \, \extcobr$ term,
        \begin{equation*}
            {\includegraphics[scale=1.2]{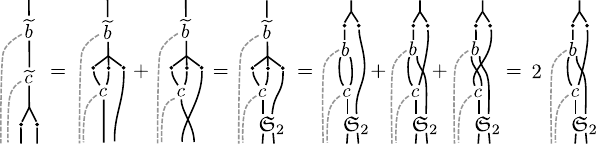}}\;\;,
        \end{equation*}
        where in the last equality we used symmetry of $\cobr$ as well as $\br \cobr = 0$ (involutivity). Now we compose the sum of these terms with $I_2$, to find 
        \begin{equation*}
            \vcenter{\hbox{\includegraphics[scale=1.2]{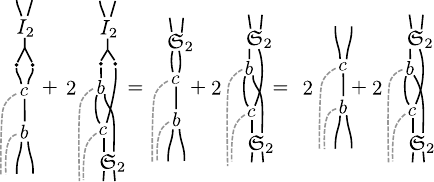}}}\;\;.
        \end{equation*}
        So $\Delta^2$ vanishes on $\Sym^{2} V$ when
        \begin{equation*}
            \vcenter{\hbox{\includegraphics[scale=1.2]{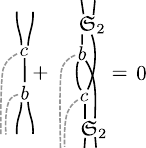}}}\;\;,
        \end{equation*}
        which is the cocycle condition. Its alternative form \eqref{eq:cocycle} follows since
        \begin{equation}\label{eq:altform}
            \vcenter{\hbox{\includegraphics[scale=1.2]{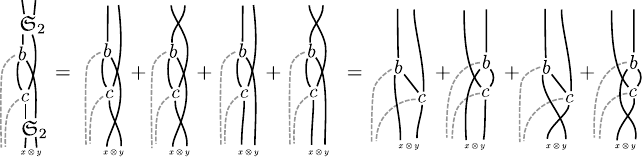}}}
        \end{equation}
        is nothing but $(-1)^{x}\mathrm{ad}_x (\cobr(y)) + (-1)^{(x+1)y}\mathrm{ad}_y (\cobr(x))$. Notice that since $[\,\extbr,\extcobr\,]$ has order 2, the involutivity and cocycle conditions are equivalent to $[\,\extbr,\extcobr\,]$ vanishing on $\Sym V$.

        \item On $\Sym^{3} V$:\\
        Now assume that $\Delta^2$ vanishes on $\Sym^{\leq 2} V$. Then, as we have seen, $\extcobr\,^2 = 0$ and $[\,\extbr,\extcobr\,] = 0$. So, for any $x,y,z \in V$, $\Delta^2(xyz) = \extbr\,^2(xyz)$, which can be written diagrammatically as
        \begin{equation*}
            \vcenter{\hbox{\includegraphics[scale=1.2]{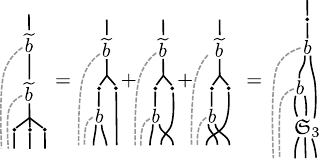}}}\;\;.
        \end{equation*}
        Composing with $I_1$, as before, we find that $\Delta^2$ vanishes on $\Sym^{3} V$ precisely when $\br(\br \otimes 1)\mathfrak{S}_3 = 0$ --- Jacobi.
    \end{itemize}
    \end{proof}

    \subsection{Lie bialgebras and shifts}
    Let $\mathscr{A}$ be a $\mathbb Z_2$-graded Lie bialgebra. Then its bracket $\operatorname{br}$ and cobracket $\delta$ are even and don't satisfy the symmetry property \eqref{eqn: symmetry properties}. However, it turns out that after a suitable shift of the grading of $\mathscr{A}$, we recover the situation from the previous section. Let us now first explain the general idea behind shifting, and then formulate our definition of an involutive Lie bialgebra with possibly shifted operations, covering both the case of a usual involutive Lie bialgebra and the case $(V, \br, \cobr)$ from previous section.

    Let us denote by $S$ the 1-dimensional graded vector space $S$ concentrated in degree 1. For brevity, we write $S^{\otimes n}$ as $S^n$ and, for uniformity, we write $S$ as $S^1$.
    \begin{defn}
        The shift of a $\mathbb Z_2$-graded vector space $\mathscr{A}$ is defined by $\mathscr{A}[k] := S^{-k}\otimes \mathscr A$. 
    \end{defn}
    Let $f \colon \mathscr{A}_1\otimes \dots \otimes \mathscr{A}_m \to \mathscr{B}_1 \otimes \dots \otimes \mathscr{B}_n$ be a morphism of degree $|f|$. We can instead see it as a degree $0$ morphism\footnote{As an illustrative example, consider a degree $1$ element of $\mathscr{B}$: it is the same as a degree $1$ morphism $\mathbb C \to \mathscr{B}$, or a degree $0$ morphism $S \to \mathscr{B}$.} $f \colon S^{|f|}\otimes\mathscr{A}_1\otimes \dots \otimes \mathscr{A}_m \to \mathscr{B}_1 \otimes \cdots \otimes \mathscr{B}_n$.
    Let us now explain how we define a morphism $f^{[k]}$ between tensor products of shifted spaces $\mathscr{A}_i[k]$, $\mathscr{B}_i[k]$.
    \begin{defn}\label{defn:shiftmorph}
        Let $f \colon S^{ |f|}\otimes\mathscr{A}_1\otimes \dots \otimes \mathscr{A}_m \to \mathscr{B}_1 \otimes \dots \otimes \mathscr{B}_n$ be a degree $0$ morphism. Then the shifted morphism
        \[ f^{[k]} \colon S^{ |f| - k(n-m)}\otimes\mathscr{A}_1[k]\otimes \dots \otimes \mathscr{A}_m[k] \to \mathscr{B}_1[k] \otimes \dots \otimes \mathscr{B}_n [k] \]
        is the degree $0$ morphism defined by the diagram
        \[  \includegraphics[scale=1.2]{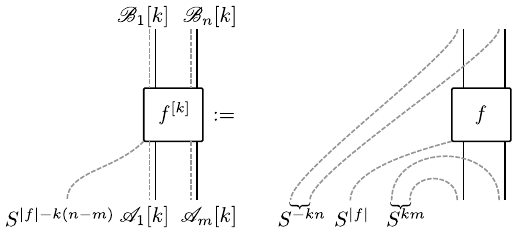}\]
        where we use the obvious isomorphism $S^{|f|-k(n-m)} \cong S^{-kn} \otimes S^{|f|} \otimes S^{km}$ (i.e. without any Koszul signs). 
    \end{defn}
    If we instead think of $f$ as a degree $|f|$ morphism  $\mathscr{A}_1\otimes \dots \otimes \mathscr{A}_m \to \mathscr{B}_1 \otimes \dots \otimes \mathscr{B}_n$, the operation $f \mapsto f^{[k]}$ changes the degree of $f$ by $-kn +km$. We note that $f^{[k]}$, the natural shift of a morphism $f$ described above, contains Koszul signs, coming from the inputs and outputs crossing the $S^{\pm k}$ dashed lines.
    \medskip
    
    Let us now turn to shifting of Lie bialgebras.
    \begin{defn}\label{defn: GLBv2} A \emph{degree $p\in \mathbb Z_2$ involutive\footnote{Of course, by dropping the condition \eqref{gb5}, we would obtain a definition of a degree $p$ Lie bialgebra. All Lie bialgebras we encounter in this work are involutive.} Lie bialgebra on a $\mathbb Z_2$-graded vector space $\mathscr A$} is given by linear maps $\operatorname{br}: \mathscr A \otimes \mathscr A \to \mathscr A$ (the \textit{bracket}) and $\delta:\mathscr A \to \mathscr A \otimes \mathscr A$ (the \textit{cobracket}) of degree $-p$ such that the following is satisfied for all homogeneous $x,y \in \mathscr A$:
        \begin{enumerate}[(i)]
            \item \label{gb1} $\operatorname{br}(x,y) = (-1)^{x \cdot y + p + 1}\operatorname{br}(y,x)$
            \item \label{gb2} $\mathrm{im}(\delta) \subset \langle x\otimes y + (-1)^{x \cdot y + p +1} y \otimes x \rangle$
            \item \label{gb3} $\operatorname{br}(\operatorname{br} \otimes 1) \mathfrak{S}_3 = 0$
            \item \label{gb4} $\mathfrak{S}_3(\delta \otimes 1) \delta = 0$
            \item \label{gb5} $\operatorname{br}\circ \, \delta = 0$
            \item \label{gb6} $\delta(\operatorname{br}(x,y)) = (-1)^{(x + p) \cdot p} \mathrm{ad}_x (\delta(y)) + (-1)^{(x+p) \cdot y + 1}\mathrm{ad}_y (\delta(x))$
        \end{enumerate}
  In the graphical language, Items \eqref{gb1} and \eqref{gb2} are variants of the identities in \eqref{eqn: symmetry properties} (interpreting $\br$ and $\cobr$ as $\operatorname{br}$ and $\delta$), with the only difference being a factor of $(-1)^{p+1}$ multiplying the RHS's. Similarly, Items \eqref{gb3} through \eqref{gb6} are variants of the identities in \eqref{eqn: GLB conditions}, where the cocycle condition \eqref{gb6} is here
\begin{equation*}
    \vcenter{\hbox{\includegraphics[scale=1.2]{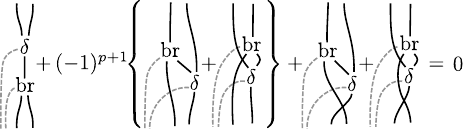}}}\;.
\end{equation*}
  
    \end{defn}
Shifting the space $\mathscr{A}$ by $p+1$ makes both $\operatorname{br}$ and $\delta$ into degree $1$ (i.e. odd) morphisms. Let us now show these combine into a BV operator.
\begin{prop}\label{prop: shiftIBL}
    Let $\mathscr A$ be a $\mathbb Z_2$-graded vector space and  let $\operatorname{br}\colon \mathscr A \otimes \mathscr A \to \mathscr A$ and $\delta:\mathscr A \to \mathscr A \otimes \mathscr A$ be degree $-p$ linear maps. Then these maps satisfy the six conditions from Definition \ref{defn: GLBv2} if and only if the shifted maps $\operatorname{br}^{[p+1]}$ and $\delta^{[p+1]}$ on $\mathscr A[p+1]$ are graded symmetric and satisfy the four equations from Proposition \ref{prop:constructing BVA}, i.e. if $\widetilde{\operatorname{br}^{[p+1]}} + \widetilde{\delta^{[p+1]}}$ defines a BV operator on $\Sym (\mathscr A[p+1])$.
\end{prop}
Denoting the shifted bracket and cobracket by $\overline{\operatorname{br}} := \operatorname{br}^{[p+1]}$ and $\overline{\delta}:= \delta^{[p+1]}$, Definition \ref{defn:shiftmorph} gives us 
\begin{equation}\label{eqn: shift of br and delta}
        \vcenter{\hbox{\includegraphics[scale=1.2]{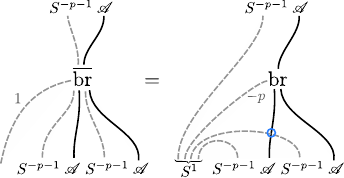}}},\quad
        \vcenter{\hbox{\includegraphics[scale=1.2]{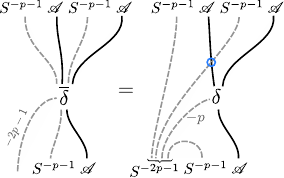}}}\;\;,
    \end{equation}
    where we have circled in blue the signs picked up by the shift.
Explicitly, given a homogeneous basis $x_\mu$ of $\mathscr A$, we can relate the components\footnote{The components are defined as usual, i.e. by $\operatorname{br}(x_\mu, x_\nu) = b_{\mu\nu}^\rho x_\rho$, where the summation over $\rho$ is implied.} $b_{\mu\nu}^\rho$ and $c_{\mu}^{\nu\rho}$ of the maps $\operatorname{br}$ and $\delta$ to the components $\overline{b}_{\mu\nu}^\rho$ and $\overline{c}_{\mu}^{\nu\rho}$ of their shifts $\overline{\mathrm{br}}$ and $\overline{\delta}$ by
    \begin{equation}\label{f,g transformation}
        \overline{b}_{\mu \nu}^{\rho} = (-1)^{\mu \cdot (p+1)}b_{\mu \nu}^{\rho}, \quad
        \overline{c}_{\mu}^{\nu \rho} = (-1)^{\nu \cdot (p+1)}c_{\mu}^{\nu \rho}.
    \end{equation}
\begin{proof}
    We want to understand the  shifted version of the six identities from Definition \ref{defn: GLBv2}. Let us thus explain how the shifting affects  composition and tensor products of morphisms. First, the shift of the flip map\footnote{I.e. the symmetry morphism $x\otimes y \to (-1)^{|x|\, |y|} y\otimes x$} $\mathscr{A}\otimes \mathscr{A} \to \mathscr{A}\otimes \mathscr{A}$ differs from the flip map $\mathscr{A}[p+1]\otimes \mathscr{A}[p+1] \to \mathscr{A}[p+1]\otimes \mathscr{A}[p+1]$. The latter differs from the former by one crossing of $S^{-p-1}$ with itself, i.e. we get an additional sign $(-1)^{p+1}$. 

    Next, it is easy to check that the shifting is compatible with composition\footnote{Composition and tensor products of morphisms of non-zero degree is defined as follows. For $f\colon S^{|f|}\otimes \mathscr{A}_n \to \mathscr{A}_o$ and $g\colon S^{|g|}\otimes \mathscr{A}_m\to \mathscr{A}_n$, we set set $f\circ g$ as $S^{|f|} \otimes S^{|g|} \otimes \mathscr{A}_m\xrightarrow{\id \otimes g} S^{|f|}  \otimes \mathscr{A}_n \xrightarrow{f}\mathscr{A}_o$. For $f' \colon S^{|f'|}\otimes \mathscr{A}_{n'}\to \mathscr{A}_{o'}$, we set $f\otimes f'$ to be
    \[ S^{|f|} \otimes  S^{|f'|}\otimes \mathscr{A}_{n}\otimes \mathscr{A}_{n'} \xrightarrow{\text{flip}}  S^{|f|}\otimes \mathscr{A}_{n} \otimes  S^{|f'|}\otimes \mathscr{A}_{n'} \xrightarrow{f \otimes f'} 
     \mathscr{A}_{o} \otimes \mathscr{A}_{o'}.\]}, i.e. for general shifting by $k$
    \[ (f\circ g)^{[k]} = f^{[k]}\circ g^{[k]}.\]
    However, for tensoring morphisms, we get a sign
    \begin{equation}\label{eq:shiftsfortensors} (f\otimes \id)^{[k]} = (-1)^{k |f|} f^{[k]} \otimes \id^{[k]}, \quad \quad (\id \otimes f)^{[k]} = (-1)^{k |f^{[k]}|}\id^{[k]} \otimes f^{[k]}. \end{equation}
    If $f$ has degree $0$ and the same number of inputs and outputs (e.g. a flip), both of the signs above vanish. In our case, for $f=\operatorname{br} \text{ or } \delta$ (which both have degree $p$) and $k=p+1$, we get
     \[ (f\otimes \id)^{[k]} =  f^{[k]} \otimes \id^{[k]}, \quad \quad (\id \otimes f)^{[k]} = (-1)^{p+1} \mathrm{id}^{[k]} \otimes f^{[k]}. \]
     In other words, we can summarize the rules as: 
     \begin{center}
         \emph{When replacing $\operatorname{br}$ and $\delta$ with $\overline{\operatorname{br}}$ and $\overline{\delta}$ in the diagrams for the six equations in Definition \ref{defn: GLBv2}, each term acquires a sign $(-1)^{p+1}$ for each crossing.}
     \end{center}
     Equations in Items \eqref{gb1} and \eqref{gb2} each have a term with one crossing; the sign cancels with $(-1)^{p+1}$ and we get the usual graded symmetry of the shifted bracket and cobracket. 
     In Items \eqref{gb3} and \eqref{gb4}, there are either $0$ or $2$ crossings in each term of the cyclic symmetriser, and thus there is no picked up sign and the equations stays the same. In Item \eqref{gb5}, there are no crossings. Finally, writing Item \eqref{gb6} as
    \begin{equation} \label{eq:pictorialcocycle}
        \vcenter{\hbox{{\includegraphics[scale=1.2]{Diagrams/024.pdf}}}}\;\;,
    \end{equation}
    we see that the terms in the curly bracket have an odd number of crossings, which cancels the sign in front. The other terms have an even numbers of crossing, and thus remain unchanged, and we recover the last equality from Proposition \ref{prop:constructing BVA} in the form \eqref{eq:altform}.
\end{proof}

\section{Graded Necklace Lie Bialgebras}\label{section: GNLBs}

    We now introduce a graded version of Schedler's necklace Lie bialgebra \cite{Schedler}.
    Fix a quiver $Q$. As usual, $Q_0$ will denote its set of vertices and $Q_1$ its set of arrows. The source and terminal vertices of an arrow $a \in Q_1$ are written as $s(a)$ and $t(a)$, respectively. Let $\overline{Q}$ be the double of $Q$. That is, $\overline{Q}_0 := Q_0$ and $\overline{Q}_1 := \bigcup_{a \in Q_1}\{a,\overline{a}\}$, where $\overline{a}$ is called the double arrow of $a$ and satisfies $t(\overline{a}) = s(a), s(\overline{a}) = t(a)$. By convention, $\overline{\overline{a}}$ is defined to be $a$. 

    We denote the path algebra of $\overline{Q}$ by $\mathcal{A}$. Recall that this is the vector space with basis the set of all paths in $\overline{Q}$ (including the ``constant path'' $e_v$ for each vertex $v \in \overline{Q}_0 = Q_0$), together with the usual multiplication as concatenation of paths. When we say ``let $a_1 \cdots a_n$ be a path'', it is understood that each $a_i$ is an \textit{arrow} (equivalently, a path of length 1), and that the path is read \textit{left to right}; so that $t(a_i) = s(a_{i+1})$, etc.
    
    Fixing $p \in \mathbb{Z}_2$, we define a $\mathbb{Z}_2$-grading on the path algebra $\mathcal{A}$ based on the numbers of new arrows in paths (which are homogeneous elements of $\mathcal{A}$). Specifically we define the degree of constant paths to be $0$, while a nonconstant path $a_1 \cdots a_n$ (with $n \geq 1$) is defined to have degree
    $
    |a_1 \cdots a_n| := p \cdot \#\{i \; \vert \; a_i \notin Q_1\}.
    $
    
    Equivalently, this grading can be induced by declaring all arrows in $Q$ to have degree $0$ and their doubles to have degree $p$, and then demanding the product in $\mathcal{A}$ to respect these degrees.
    
    We now consider the quotient of $\mathcal{A}$ by its graded-commutator subspace, $[\mathcal{A},\mathcal{A}]$:
    $$
    \mathscr{A} := \mathcal{A} / [\mathcal{A},\mathcal{A}].
    $$
    We can think of $\mathscr{A}$ as the vector space generated by paths in $\overline{Q}$ subject to relations of the form 
    \begin{equation} \label{eq:cyclicpropexplicit}a_1 \cdots a_n = (-1)^{a_1 \cdot (a_2 + \dots +a_n)} a_2 \cdots a_n a_1 = \dots = (-1)^{(a_1 + \dots + a_{n-1}) \cdot a_n} a_n a_1 \cdots a_{n-1},\end{equation}
    where $a_1\cdots a_n$ is an arbitrary nonconstant path.
    It is then clear that $\mathscr{A}$ is spanned by \textit{closed} paths in $\overline{Q}$ (this includes the constant paths), and that a closed path in $\overline{Q}$ is either 0 in $\mathscr{A}$ (as a consequence of the relations) or uniquely determines a 1 dimensional subspace of $\mathscr{A}$ which does not depend on the cyclic path's endpoints. As a result, we introduce notation to treat different choices of endpoints on an equal footing:

    For a non-constant path $a_1\cdots a_n$ in $\overline{Q}$ (which we label by $A$ when forgetting endpoints entirely) we write $A_k := a_k \cdots a_{k-1}$, so that the original path is $A_1$. In the quotient $\mathscr{A}$ we have $A_k = (-1)^{\varepsilon^A_{kl}}A_l$, where we have introduced the notation $\varepsilon_{kl}^A$ for the relative sign, which is explicitly $|a_k \cdots a_{l-1}||a_l \cdots a_{k-1}|$. Notice that we necessarily have
    \begin{equation}\label{eqn:epsilon_identity}
        \varepsilon_{kl}^A = \varepsilon_{km}^A + \varepsilon_{ml}^A.
    \end{equation}
    Furthermore, we define $A_{ij} \in \mathcal A$ by
    $$
    A_{ij} := 
    \begin{cases}
        e_{t(a_i)}, \; \text{if $j = i+1 \mod n$};\\
        a_{i+1} \cdots a_{j-1}, \; \text{otherwise.}
    \end{cases}
    $$
    The above notation can be summarised as follows:
    \begin{equation*}
        {\includegraphics[scale=1.1]{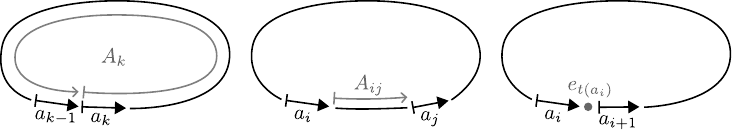}}.
    \end{equation*}
    
    As in Schedler's case, we define an ``indicator'' $\langle-,-\rangle\colon \overline{Q}_1 \times \overline{Q}_1 \to \mathbb{C}$ as follows: 
    \begin{equation*}
        \langle a, b \rangle :=
        \begin{cases}
            0, \; \text{if $b \neq \overline{a}$};\\
            1, \; \text{if $a \in Q_1$ and $b = \overline{a}$};\\
            (-1)^{p+1}, \; \text{if $a \notin Q_1$ and $b = \overline{a}$}.
        \end{cases}
    \end{equation*}
    This indicator\footnote{The rule $\langle a, b \rangle = (-1)^{p+1}\langle b, a \rangle$ can be motivated as follows: the indicator corresponds to a symplectic form $\sum_a \mathrm{d} a\, \mathrm{d} \overline{a}$, where both differentials are odd iff $p$ is even.} is a degree $-p$ map, and using it we define degree $-p$ linear maps
    \begin{equation*}
        \operatorname{br}\colon \mathcal{A} \otimes \mathcal{A} \to \mathscr{A}, \quad
        \delta\colon \mathcal{A} \to \mathscr{A} \otimes \mathscr{A}
    \end{equation*}
    on the basis of paths as follows:
    \begin{itemize}
        \item $\operatorname{br}(x,y)$ is defined to be zero if $x$ or $y$ has length $0$ (i.e. is a constant path). Otherwise, for non-constant paths $A$ and $B$, we define
        \begin{equation}\label{eqn: Bracket defn}
            \operatorname{br}(A_k,B_l)
            =
            \sum_{i}^{A}\sum_{j}^{B}(-1)^{\varepsilon_{ki}^{A} + \varepsilon_{lj}^B + A_{ii} \cdot b_j}\langle a_i, b_j \rangle A_{ii}B_{jj},
        \end{equation}
        where we have stopped being so explicit as to whether we are viewing paths as elements of $\mathcal{A}$ or $\mathscr{A}$ --- this will be clear from context. We can write the $ij$'th term of \eqref{eqn: Bracket defn} diagrammatically as
        \begin{equation*}
            {\includegraphics[scale=1.2]{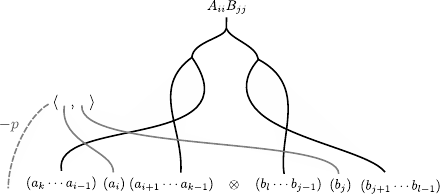}}\;\;.
        \end{equation*}
        Due to \eqref{eqn:epsilon_identity}, we see that $\operatorname{br}$ descends to a map $\mathscr{A} \otimes \mathscr{A} \to \mathscr{A}$ which we denote by $\operatorname{br}$ also.

        \item $\delta(x)$ is defined to be zero if $x$ has length $\leq 1$. Otherwise, for a path $A$ of length $\geq 2$ we define
        \begin{equation}\label{eqn: delta usual-notation expression}
            \delta(A_k) 
            =
            \frac{1}{2}\sum_{i,j}^{A} (-1)^{\varepsilon_{ki}^A + A_{ij} \cdot a_j} \langle a_i, a_j \rangle A_{ij} \otimes A_{ji},
        \end{equation}
        and the $ij$'th term is written diagrammatically as
        \begin{equation*}
            {\includegraphics[scale=1.2]{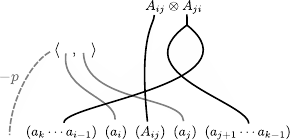}}\;\;,
        \end{equation*}
        up to a factor of $\frac{1}{2}$. Again, $\delta$ descends to a map $\mathscr{A} \to \mathscr{A} \otimes \mathscr{A}$, which we denote by $\delta$ also.
    \end{itemize}

    \begin{prop}\label{prop:GSNLB}
        \textit{The tuple $(\mathscr{A},\operatorname{br},\delta)$, which we will call the \emph{graded Schedler necklace Lie bialgebra}, is a degree $p$ involutive Lie bialgebra.}
    \end{prop}
    \begin{proof}
        As $\operatorname{br}$ and $\delta$ have degree $-p$, we must check that the six conditions of Definition \ref{defn: GLBv2} are satisfied. This is simply a tedious calculation. Introducing some new notation to help us, we prove the $p=0$ case (supplying details to Schedler's calculation) in Appendix \ref{Appendix:Cals} --- the method for general $p$ is identical, albeit with more sign bookkeeping. 
    \end{proof}
    As explained in Propositions \ref{prop:constructing BVA} and \ref{prop: shiftIBL}, we can pass to the shifted space $\mathscr{A}[p+1] = S^{-p-1} \otimes \mathscr{A}$ to get a BV operator on $\Sym (\mathscr{A}[p+1])$. 
    \begin{defn}\label{def:necklaceBV}
        Choose $\hbar \in \mathbb{C}^\times$.
        Let $\mathscr{B} := \Sym (\mathscr{A}[p+1])$ be the Batalin-Vilkovisky algebra with BV operator $\Delta_\hbar := \widetilde{\operatorname{br}} + \hbar \widetilde{{\delta}}$. Here, both terms are given as extensions to differential operators of the shifted bracket $\overline{\operatorname{br}}$ and shifted cobracket $\overline{\delta}$ --- see \eqref{eqn: shift of br and delta}.
    
    \end{defn}
    Note that, when shifting $\operatorname{br}$ and $\delta$, the Koszul signs in \eqref{eqn: shift of br and delta} are always $+1$ since the degree of any homogeneous element $x \in \mathscr{A}$ is a multiple of $p$, and hence $(-1)^{x \cdot (-p-1)} = 1$.
    In the usual case $p=0$, we get a BV operator extending the differential on the (dual) Chevalley-Eilenberg complex $\Sym (\mathscr{A}[1])\cong \bigwedge\mathscr{A}$.


    
\section{BV Algebra on a Representation Variety}\label{section: BV Algebra on a Representation Variety}

In this section we construct another BV algebra, $(\widetilde{\mathcal{O}},\widetilde{\Delta})$, which we relate to the BV algebra $(\mathscr{B},\Delta_\hbar)$ constructed from a quiver $Q$, parameter $p$, and constant $\hbar$. The input data for the present construction is a choice, for each vertex $i$ of $Q$, of a $\mathbb{Z}_2$-graded vector space $V_i$ \emph{with a degree $-p-1$ invertible linear map $\iota^i:V_i \to V_i$} such that $(\iota^i)^2$ is equal to $\lambda \mathrm{id}_{V_i}$ for some constant $\lambda \in \mathbb C$ independent of the vertex $i$. We can then consider the following variant (Equation \eqref{eq:iotaRV}) of the quiver representation variety: the vector space of linear maps which intertwine the maps $\iota$. This vector space will have a non-degenerate, degree $-1$ pairing given in Equation \eqref{eqn: pairing}. The BV algebra  $(\widetilde{\mathcal{O}},\widetilde{\Delta})$ is the algebra of polynomial functions on this representation variety, and the BV operator is the constant-coefficients, second-order differential operator given by the inverse of the pairing.

\subsection{Coordinate-free construction}
To begin, let $V = \{V_i\}_{i \in \overline{Q}_0}$ be a collection of finite dimensional $\mathbb{Z}_2$-graded vector spaces --- one at each vertex of $\overline{Q}$. For each arrow $a \in \overline{Q}_1$ we get a $\mathbb{Z}_2$-graded vector space $\mathrm{H}^a := \underline{\operatorname{Hom}}(V_{s(a)},V_{t(a)})$ of the set of \textit{all} linear mappings $V_{s(a)} \to V_{t(a)}$ (not just those that are even). We then consider the following shifted version of the representation variety
\begin{equation*}
    \mathcal{R}(\overline{Q},{V}) := {\bigoplus}_{a \in \overline{Q}_1}\mathrm{H}^a \otimes S^{-a},
\end{equation*}
where, as before, $S^{-a}$ denotes the homogeneous 1-dimensional vector space of degree $-|a|$ (so degree $-p$ if $a \notin Q_1$, and $0$ otherwise). With this we define the (graded) symmetric algebra
\begin{equation*}
    \mathcal{O} := \mathcal{O}(\mathcal{R}(\overline{Q},{V})) = \Sym(\mathcal{R}(\overline{Q},{V})^*).
\end{equation*}

Now, for each $i \in \overline{Q}_0$, let $\iota^i \in \underline{\operatorname{End}}(V_i)$ be a degree $-p-1$ invertible map whose square is equal to $\lambda \mathrm{id}_{V_i}$, where $\lambda \in \mathbb{C}$ is a constant that is independent of the vertex $i$. Then for each $a \in \overline{Q}_1$ we have a subspace $\mathrm{H}_{\iota}^a \subset \mathrm{H}^a$ of maps that (graded) commute with $\iota$. That is, if $f \in \mathrm{H}^a$ is homogeneous, then
\begin{equation*}
    f \in \mathrm{H}_{\iota}^a \iff [\iota,f] = \iota \circ f - (-1)^{\iota \cdot f} f \circ \iota = \iota^{t(a)} \circ f - (-1)^{\iota \cdot f} f \circ \iota^{s(a)} = 0.
\end{equation*}
Here we have seen an example of us writing $\iota$ instead of the more correct $\iota^i$ for some vertex $i$. We frequently do this to ease notation when the vertex is clear from context. 
The subspace of all maps intertwining the $\iota$'s in this way will be the representation variety of interest for us.
\begin{defn} Let 
\begin{equation}\label{eq:iotaRV}
    \mathcal{R}_\iota(\overline{Q},{V}) := {\bigoplus}_{a \in \overline{Q}_1}\mathrm{H}^a_{\iota} \otimes S^{-a},
\end{equation}
and denote its algebra of formal polynomial functions as
\begin{equation*}
    \widetilde{\mathcal{O}} := \mathcal{O}(\mathcal{R}_\iota(\overline{Q},{V})) = \Sym(\mathcal{R}_\iota(\overline{Q},{V})^*).
\end{equation*}
\end{defn}
\begin{rem}\label{rem:subandquot}
The algebras $\mathcal{O}$ and $\widetilde{\mathcal{O}}$ defined above are related by $\widetilde{\mathcal{O}} \cong \mathcal{O}/I$, where $I$ is the ideal generated by the annihilator $\mathrm{Ann}_{{\mathcal{R}}_\iota(\overline{Q},{V})} \subset \mathcal{R}(\overline{Q},{V})^*$.
\end{rem}
We now aim to define a BV operator on $\widetilde{\mathcal{O}}$. To this end, for each $a \in Q_1$ we define a degree 0 pairing
\begin{equation}\label{eqn: pairing}
    S^{-1} \otimes \mathrm{H}_{\iota}^a \otimes \mathrm{H}_{\iota}^{\overline{a}} \otimes S^{-p} \to \mathbb{C};\quad s^{-1} \otimes \phi \otimes \psi \otimes s^{-p} \mapsto \operatorname{str}(\iota \circ \phi \circ \psi),
\end{equation}
where $s^k$ denotes unity in $S^k$, and $\operatorname{str}$ denotes the usual supertrace, acting on $\iota \circ \phi \circ \psi \colon V_{t(a)} \to V_{t(a)}$. In Proposition \ref{prop: non-degeneracy} we prove that this pairing is non-degenerate. For this it helps to define a linear map $B_{\iota}\colon\mathrm{H}^{a} \to \mathrm{H}^{a}$  by $B_{\iota}(f) = \iota \circ f + (-1)^{\iota \cdot f + \iota} f \circ \iota$, for homogeneous\footnote{Note the strange sign. The map $B_\iota$ has surprising algebraic properties and appears again in Lemma \ref{lem: intertwining 1}.} $f$. In fact, we define such a map for each arrow $a \in \overline{Q}_1$, and denote them all by $B_{\iota}$.
\begin{lem}\label{lem: B_iota image}
    \textit{The map $B_{\iota}\colon\mathrm{H}^{a} \to \mathrm{H}^{a}$ satisfies $B_{\iota}(\mathrm{H}^a) = \mathrm{H}_{\iota}^{a}$.}
\end{lem}
\begin{proof}
    Let $f \in \mathrm{H}^{a}$.
    A quick calculation shows that $[\iota,B_{\iota}(f)] = [\iota^2,f] = 0$, where the last equality follows since $\iota^2$ is central. Thus $B_{\iota}(\mathrm{H}^a) \subset \mathrm{H}^{a}_{\iota}$.
    On the other hand, if $\psi \in \mathrm{H}_{\iota}^{a}$, then $\psi = B_{\iota} (\frac{1}{2} \iota^{-1} \circ \psi)$, so that $B_{\iota}(\mathrm{H}^{a}) \supset \mathrm{H}_{\iota}^{a}$.
\end{proof}
\begin{prop}\label{prop: non-degeneracy}
    \textit{The pairing \eqref{eqn: pairing} is non-degenerate.}
\end{prop}
\begin{proof}
    We need to show that 
    \begin{equation}\label{eqn: non-deg condition 1}
        \phi \in \mathrm{H}^{a}_{\iota} \;\; \text{and} \; \; \operatorname{str}(\iota \circ \phi \circ \psi)=0 \; \text{for all} \; \psi \in \mathrm{H}^{\overline{a}}_{\iota}
        \implies
        \phi = 0,
    \end{equation}
    and that 
    \begin{equation}\label{eqn: non-deg condition 2}
        \psi \in \mathrm{H}^{\overline{a}}_{\iota} \;\; \text{and} \; \; \operatorname{str}(\iota \circ \phi \circ \psi)=0 \; \text{for all} \; \phi \in \mathrm{H}^{a}_{\iota}
        \implies
        \psi = 0.
    \end{equation}
    Assume the premise of \eqref{eqn: non-deg condition 1}. By Lemma \ref{lem: B_iota image}, we have that $\operatorname{str}(\iota \circ \phi \circ B_{\iota}(f)) = 0$ for all $f \in \mathrm{H}^{\overline{a}}$. By the fact that $\phi$ commutes with $\iota$ and by cyclicity of $\operatorname{str}$, we find that $\operatorname{str}(\iota \circ \phi \circ B_{\iota}(f)) = 2(-1)^{\iota \cdot \phi} \operatorname{str}(\iota^2 \circ \phi \circ f)$, and hence that $\operatorname{str}(\iota^2 \circ \phi \circ f) = 0$ for all $f \in \mathrm{H}^{\overline{a}}$. As the regular supertrace pairing is non-degenerate, it follows that $\iota^2 \circ \phi = 0$, and thus $\phi = 0$. This proves \eqref{eqn: non-deg condition 1}. \eqref{eqn: non-deg condition 2} is similar.
\end{proof}
Since the pairing is non-degenerate, there exists a unique linear map $t^a\colon \mathbb{C} \to \mathrm{H}_{\iota}^{\overline{a}} \otimes S^{-p} \otimes S^{-1} \otimes \mathrm{H}_{\iota}^{a}$ (which can be viewed as an element of $\mathrm{H}_{\iota}^{\overline{a}} \otimes S^{-p} \otimes S^{-1} \otimes \mathrm{H}_{\iota}^{a}$) inverse to this pairing. That is, such that
\begin{equation}\label{eqn: snake}
    \vcenter{\hbox{\includegraphics[scale=1.2]{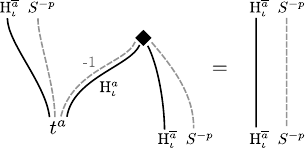}}}\;\;,
\end{equation}
where the diamond denotes the pairing \eqref{eqn: pairing}.

\begin{rem}
    Equation \eqref{eqn: snake} is the condition that $t^a$ is left inverse to the pairing. That $t^a$ is right inverse to the pairing follows from the fact that $\mathrm{H}_{\iota}^a$ and $\mathrm{H}_{\iota}^{\overline{a}}$ have the same dimension: a left inverse square matrix to a square matrix is the two-sided inverse.
\end{rem}

We use the inverses $\{t^a\}_{a \in Q_1}$ to define a degree zero linear map
\begin{equation*}
    t:S^1 \otimes 
    \underbrace{
    ({\bigoplus}_{a \in Q_1} \mathrm{H}_{\iota}^{\overline{a}} \otimes S^{-p})^*}_
    {\subset \; {\mathcal{R}}_\iota(\overline{Q},{V})^*}
    \otimes 
    \underbrace{({\bigoplus}_{a \in Q_1} \mathrm{H}_{\iota}^a)^*}_
    {\subset \; {\mathcal{R}}_\iota(\overline{Q},{V})^*}
    \to \mathbb{C}
\end{equation*}
as follows: on the subspace $S^1 \otimes S^p \otimes (\mathrm{H}_{\iota}^{\overline{a}})^* \otimes (\mathrm{H}_{\iota}^a)^*$ (with $a \in Q_1$) the map $t$ is given by
\begin{equation*}
    \vcenter{\hbox{\includegraphics[scale=1.2]{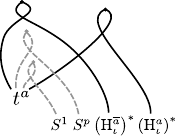}}}\;\;,
\end{equation*}
where the dots denote the usual pairing between a vector space and its dual. Meanwhile, $t$ is defined to vanish on the subspaces
$S^1 \otimes S^p \otimes (\mathrm{H}_{\iota}^{\overline{a}})^* \otimes (\mathrm{H}_{\iota}^b)^*$ with $a \neq b$ (for $a,b \in Q_1$).

Essentially by graded symmetrisation, we can use $t$ to define a degree zero linear map
\begin{equation*}
    \widetilde{\Delta}\colon S^1 \otimes \Sym^2({\mathcal{R}}_\iota(\overline{Q},{V})^*) \to \mathbb{C}:
\end{equation*}
Let $x,y \in {\mathcal{R}}_\iota(\overline{Q},{V})^*$. If $x \in ({\bigoplus}_{a \in Q_1} \mathrm{H}_{\iota}^{\overline{a}} \otimes S^{-p})^*$ and $y \in ({\bigoplus}_{a \in Q_1} \mathrm{H}_{\iota}^a)^*$, then $\widetilde{\Delta}(s^1 \otimes xy) := (-1)^\iota t(s^1 \otimes x \otimes y)$. On the other hand, if $x,y \in ({\bigoplus}_{a \in Q_1} \mathrm{H}_{\iota}^{\overline{a}} \otimes S^{-p})^*$ or $x,y \in ({\bigoplus}_{a \in Q_1} \mathrm{H}_{\iota}^a)^*$, then $\widetilde{\Delta}(s^1 \otimes xy) := 0$.

\begin{defn}
Consider $\widetilde{\Delta}$ as a degree $1$ map
\begin{equation*}
    \widetilde{\Delta}\colon \Sym^2({\mathcal{R}}_\iota(\overline{Q},{V})^*) \to \mathbb{C}.
\end{equation*}
    Denote by the same symbol the degree 1 second-order differential operator on $\widetilde{\mathcal{O}}$ given by extending $\widetilde{\Delta}$. In particular, $\widetilde{\Delta}$ vanishes on $\Sym^{\leq 1}({\mathcal{R}}_\iota(\overline{Q},{V})^*)$.
\end{defn}
This makes $(\widetilde{\mathcal{O}}, \widetilde{\Delta})$ into a BV algebra:
 the third-order operator $\widetilde{\Delta}^2$ vanishes on $\Sym^{\leq 3}({\mathcal{R}}_\iota(\overline{Q},{V})^*)$ and thereby on $\widetilde{\mathcal{O}}$.

\subsection{Realisation in Coordinates}
We now build up an understanding of this BV algebra in coordinates, and begin by choosing bases and setting some notation. Let us also define the space
\begin{equation*}
    \underline{\operatorname{End}} {V}
    :=
    \underline{\operatorname{End}}({\bigoplus}_{i\in \overline{Q}_0} V_i)
    \cong
    {\bigoplus}_{i,j \in \overline{Q}_0} \underline{\operatorname{Hom}}(V_i,V_j),
\end{equation*}
which will be useful for us later.

Choose a homogeneous basis $\{v_{\mu}^{(i)}\}$ for $V_i$, for each vertex $i \in \overline{Q}_0$. The degree of $v_{\mu}^{(i)}$ is denoted by $\mu$. Now let $\tensor{(h^{i \to j} )}{^{\mu}_{\nu}} \in \underline{\operatorname{Hom}}(V_i,V_j)$ be the degree $\nu - \mu$ linear map defined by $\tensor{(h^{i \to j} )}{^{\mu}_{\nu}} (v_{\rho}^{(i)}) = \delta_{\rho}^{\mu} v_{\nu}^{(j)}$. Then $\{\tensor{(h^{i \to j} )}{^{\mu}_{\nu}}\}$ is a homogeneous basis for $\underline{\operatorname{Hom}}(V_i,V_j)$. Furthermore, for each arrow $a\in\overline{Q}_1$, we define $\tensor{(h^{a} )}{^{\mu}_{\nu}} \in \mathrm{H}^a$ to be $\tensor{(h^{s(a) \to t(a)} )}{^{\mu}_{\nu}}$ together with the label $a$. And so we have that
\begin{equation*}
    \{\underbrace{\tensor{(h^{a} )}{^{\mu}_{\nu}}}_{\text{deg. $\nu - \mu$}} \}_{a \in \overline{Q}_1} \; \text{is a basis for} \; {\bigoplus}_{a \in \overline{Q}_1} \mathrm{H}^a, \; \text{and} \;
    \{
    \underbrace{
    \tensor{(h^{i \to j} )}{^{\mu}_{\nu}}
    }_
    {\text{deg. $\nu - \mu$}}
    \}_{i,j \in \overline{Q}_0} \; \text{is a basis for} \; \underline{\operatorname{End}}{V}.
\end{equation*}
\begin{equation*}
    \implies \{\underbrace{\tensor{(h^{a} )}{^{\mu}_{\nu}} \otimes s^{-a}}_{\text{deg. $\nu - \mu - a$}} \}_{a \in \overline{Q}_1} \; \text{is a basis for} \; {\bigoplus}_{a \in \overline{Q}_1} \mathrm{H}^a \otimes S^{-a} = \mathcal{R}(\overline{Q},{V}).
\end{equation*}
We then let $\tensor{(x^a )}{_{\mu}^{\nu}} \in (\mathrm{H}^a)^*$ be the degree $\mu - \nu$ dual vector with $\tensor{(x^a )}{_{\mu}^{\nu}} \left( \tensor{(h^{a} )}{^{\rho}_{\sigma}} \right) = \delta_{\mu}^{\rho} \delta_{\sigma}^{\nu}$, thereby getting a dual basis
\begin{equation*}
    \{\underbrace{\tensor{s^{a} \otimes (x^{a} )}{_{\mu}^{\nu}}}_{\text{deg. $a + \mu - \nu$}} \}_{a \in \overline{Q}_1} \; \text{for} \; \mathcal{R}(\overline{Q},{V})^*, \; \text{which generates} \; \mathcal{O}.
\end{equation*}
This set also generates $\widetilde{\mathcal{O}}$, under the projection $\mathcal{O} \twoheadrightarrow \widetilde{\mathcal{O}}$.

With these bases, we can now consider the components of different maps. For a linear map $f \in \underline{\operatorname{Hom}}(V_i,V_j)$, for example, we write $f = \tensor{f}{_{\mu}^{\nu}} \tensor{(h^{i \to j} )}{^{\mu}_{\nu}}$. For the unique $t^a \in \mathrm{H}_{\iota}^{\overline{a}} \otimes S^{-p} \otimes S^{-1} \otimes \mathrm{H}_{\iota}^{a} \subset \mathrm{H}^{\overline{a}} \otimes S^{-p} \otimes S^{-1} \otimes \mathrm{H}^{a}$ inverse to the pairing \eqref{eqn: pairing} (for some arrow $a \in Q_1$), we write
\begin{equation}\label{eqn: t expansion}
    t^a = (t^a)_{\alpha \gamma}^{\beta \delta} \tensor{(h^{\overline{a}} )}{^{\alpha}_{\beta}} \otimes s^{-p} \otimes s^{-1} \otimes \tensor{(h^{a} )}{^{\gamma}_{\delta}}.
\end{equation}

\begin{prop}
    \textit{The components of $t^a$ are given by}
    \begin{equation}\label{eqn: t-components}
        (t^a)_{\alpha \gamma}^{\beta \delta} =
        \frac{1}{2}(-1)^{\alpha}
        \left[
        \tensor{((\iota^{t(a)})^{-1})}{_{\alpha}^{\delta}} \delta_{\gamma}^{\beta} + (-1)^{\iota \cdot (\alpha + \beta + 1)} \delta_{\alpha}^{\delta}\tensor{((\iota^{s(a)})^{-1})}{_{\gamma}^{\beta}}
        \right].
    \end{equation}
\end{prop}
\begin{proof}
    Assume that $t^a$ is given by \eqref{eqn: t-components}. By definition, this $t^a$ lies in $\mathrm{H}^{\overline{a}} \otimes S^{-p} \otimes S^{-1} \otimes \mathrm{H}^{a}$, but it is not known a priori whether it lies in $\mathrm{H}_{\iota}^{\overline{a}} \otimes S^{-p} \otimes S^{-1} \otimes \mathrm{H}_{\iota}^{a}$. So, to prove that the ansatz \eqref{eqn: t-components} is correct, we must show that
    \begin{enumerate}
        \item $t^a$ lies in $\mathrm{H}_{\iota}^{\overline{a}} \otimes S^{-p} \otimes S^{-1} \otimes \mathrm{H}_{\iota}^{a}$, and
        \item $t^a$ satisfies the identity \eqref{eqn: snake}.
    \end{enumerate}
    Equivalently, that
    \begin{enumerate}
        \item $(C_\iota \otimes 1_{S^{-p} \otimes S^{-1} \otimes \mathrm{H}^{a}}) (t^a) = 0$ and $(1_{\mathrm{H}^{\overline{a}} \otimes S^{-p} \otimes S^{-1}} \otimes C_{\iota})(t^a) = 0$, where $C_\iota$ is the commutator with $\iota$, and then that
        \item $\left(1_{\mathrm{H}_{\iota}^{\overline{a}} \otimes S^{-p}} \otimes \operatorname{str}(\iota \circ - \circ -)\right)
        \underbrace{\left(t^a \otimes B_{\iota}(\tensor{(h^{\overline{a}} )}{^{\mu}_{\nu}}) \otimes s^{-p}\right)}
        _{\in \mathrm{H}_{\iota}^{\overline{a}} \otimes S^{-p} \otimes S^{-1} \otimes \mathrm{H}_{\iota}^{a} \otimes \mathrm{H}_{\iota}^{\overline{a}} \otimes S^{-p}}
        = B_{\iota}(\tensor{(h^{\overline{a}} )}{^{\mu}_{\nu}}) \otimes s^{-p}$
        for all $\mu,\nu$, and where $\operatorname{str}(\iota \circ - \circ -)$ denotes the pairing \eqref{eqn: pairing}. Here we have used Lemma \ref{lem: B_iota image}.
    \end{enumerate}
    These are tedious but straightforward calculations which we do not detail here.
\end{proof}
It is convenient to drop the arrow label $a$ in \eqref{eqn: t-components}, defining the components
\begin{equation}\label{eqn: t-components generalisation}
    t_{\alpha \gamma}^{\beta \delta} :=
    \frac{1}{2}(-1)^{\alpha}
    \left[
    \tensor{(\iota^{-1})}{_{\alpha}^{\delta}} \delta_{\gamma}^{\beta} + (-1)^{\iota \cdot (\alpha + \beta + 1)} \delta_{\alpha}^{\delta}\tensor{(\iota^{-1})}{_{\gamma}^{\beta}}
    \right],
\end{equation}
where, for example, it is understood that $\tensor{(\iota^{-1})}{_{\mu}^{\nu}} = \tensor{((\iota^{i})^{-1})}{_{\mu}^{\nu}}$ if the indices $\mu$ and $\nu$ are labeling the basis $\{v_{\mu}^{(i)}\}$ of $V_i$ --- so that there is no ambiguity\footnote{Note that \eqref{eqn: t-components generalisation} is implicitely a generalisation of \eqref{eqn: t-components}, akin to defining ``$(t^a)_{\mu \rho}^{\nu \sigma}$'' for all $a \in \overline{Q}_1$, not just $a \in Q_1$.}.
These satisfy $t_{\gamma \alpha}^{\delta \beta} = (-1)^{(\alpha + \beta)\cdot(\iota + 1)}t_{\alpha \gamma}^{\beta \delta}$.

As explained before, the inverses $\{t^a\}_{a \in Q_1}$ define a degree zero linear map
\begin{equation*}
    t:S^1 \otimes 
    \underbrace{
    ({\bigoplus}_{a \in Q_1} \mathrm{H}_{\iota}^{\overline{a}} \otimes S^{-p})^*}_
    {\subset \; {\mathcal{R}}_\iota(\overline{Q},{V})^*}
    \otimes 
    \underbrace{({\bigoplus}_{a \in Q_1} \mathrm{H}_{\iota}^a)^*}_
    {\subset \; {\mathcal{R}}_\iota(\overline{Q},{V})^*}
    \to \mathbb{C}.
\end{equation*}
Denoting the images of $\tensor{(x^{{a}})}{_{\mu}^{\nu}}$ under the projection $(\mathrm{H}^a)^* \twoheadrightarrow (\mathrm{H}^a_{\iota})^*$ by $\tensor{(x^{{a}})}{_{\mu}^{\nu}}$ also\footnote{In $(\mathrm{H}^a_{\iota})^*$, these coordinates $x$ satisfy a linear relation, given by the vanishing of $Y^\nu_\mu$ in \eqref{eq:Ymunu}.}, this map sends $s^1 \otimes s^p \otimes \tensor{(x^{\overline{a}})}{_{\mu}^{\nu}} \otimes \tensor{(x^b)}{_{\rho}^{\sigma}}$ to 0 when $a \neq b$ (for $a,b \in Q_1$), and sends $s^1 \otimes s^p \otimes \tensor{(x^{\overline{a}})}{_{\mu}^{\nu}} \otimes \tensor{(x^a)}{_{\rho}^{\sigma}}$ (for $a \in Q_1$) to
\begin{equation*}
    \vcenter{\hbox{\includegraphics[scale=1.2]{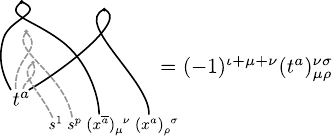}}}\;\;,
\end{equation*}
as a calculation shows.

The resulting degree 1 map $\widetilde{\Delta}\colon \Sym^2({\mathcal{R}}_\iota(\overline{Q},{V})^*) \to \mathbb{C}$ is determined by (for $a,b \in Q_1$)
\begin{itemize}
    \item $\widetilde{\Delta}\left(\left(\tensor{(x^{a})}{_{\mu}^{\nu}} \right) \cdot  \left( \tensor{(x^{b})}{_{\rho}^{\sigma}} \right) \right) = 0 = \widetilde{\Delta}\left( \left( s^p \otimes \tensor{(x^{\overline{a}})}{_{\mu}^{\nu}} \right) \cdot \left(\tensor{s^p \otimes (x^{\overline{b}})}{_{\rho}^{\sigma}} \right) \right)$, and
    \item $\widetilde{\Delta}\left( \left( s^p \otimes \tensor{(x^{\overline{a}})}{_{\mu}^{\nu}} \right) \cdot  \left(\tensor{(x^{b})}{_{\rho}^{\sigma}} \right) \right) = (-1)^{\mu + \nu} t_{\mu\rho}^{\nu \sigma} \delta^{ab}$.
\end{itemize}
Using the symmetry property $t_{\gamma \alpha}^{\delta \beta} = (-1)^{(\alpha + \beta)\cdot(\iota + 1)}t_{\alpha \gamma}^{\beta \delta}$, we can write
\begin{equation*}
    \widetilde{\Delta}\left( \left( \tensor{(x^{a})}{_{\mu}^{\nu}} \right) \cdot \left( \tensor{s^p \otimes (x^{\overline{b}})}{_{\rho}^{\sigma}} \right) \right) = (-1)^{\iota + (\mu + \nu)\cdot \iota}t_{\mu \rho}^{\nu \sigma}\delta^{ab}.
\end{equation*}
Hence we can write the action of the operator $\widetilde{\Delta}$ on $\Sym^2({\mathcal{R}}_\iota(\overline{Q},{V})^*)$ compactly as
\begin{equation}\label{eqn: delta-tilde, explicitely}
    \widetilde{\Delta}\left( \left( s^{a} \otimes \tensor{(x^{{a}})}{_{\alpha}^{\beta}} \right) \cdot  \left(\tensor{s^{b} \otimes (x^{{b}})}{_{\gamma}^{\delta}} \right) \right) 
    = \langle a,b \rangle (-1)^{\iota + (\alpha + \beta)(1 + b)}t_{\alpha \gamma}^{\beta \delta},
\end{equation}
where $a$ and $b$ are now arbitrary arrows in $\overline{Q}$. In other words, we have the following.
\begin{prop}\label{prop:BVRepexp}
\textit{ Denote $s^{a} \otimes \tensor{(x^a)}{_{\mu}^{\nu}}$ by $\tensor{(X^a)}{_{\mu}^{\nu}}$. The BV operator $\widetilde{\Delta}$ of the BV algebra $(\widetilde{\mathcal{O}},\widetilde{\Delta})$ is determined by}
\begin{equation}\label{eqn: BV operator coords}
    \widetilde{\Delta}\left( \tensor{(X^{{a}})}{_{\alpha}^{\beta}} \cdot \tensor{(X^{{b}})}{_{\gamma}^{\delta}}\right)
    = \frac{1}{2} \langle a,b \rangle (-1)^{\iota + \beta + b(\alpha + \beta)}
    \left[
    \tensor{((\iota^{s(a)})^{-1})}{_{\alpha}^{\delta}} \delta_{\gamma}^{\beta} + (-1)^{\iota \cdot (\alpha + \beta + 1)} \delta_{\alpha}^{\delta}\tensor{((\iota^{s(b)})^{-1})}{_{\gamma}^{\beta}}
    \right].
\end{equation}
\end{prop}

\section{A BV algebra morphism} \label{section: A BV algebra morphism}
Let us summarise where we are. Given data $(Q,p,\hbar)$, where $Q$ is a quiver, $p \in \mathbb{Z}_2$ is a parameter, and $\hbar \in \mathbb{C}^\times$ is a constant, we get a BV algebra $(\mathscr{B},\Delta_\hbar)$ coming from the associated degree $p$ necklace Lie bialgebra $(\mathscr{A},\mathrm{br},\delta)$. Given further data $(V,\iota)$, where $V = \{V_i\}_{i \in \overline{Q}_0}$ is a collection of $\mathbb{Z}_2$-graded vector spaces, and $\iota = \{\iota^i\}_{i \in \overline{Q}_0}$ is a collection of degree $-p-1$ invertible linear endomorphisms $\iota^i \in \underline{\mathrm{End}}{V_i}$ satisfying $(\iota^i)^2 = \lambda \mathrm{id}_{V_i}$ for some constant $\lambda \in \mathbb{C}$ independent of vertex, we get a BV algebra $(\widetilde{\mathcal{O}},\widetilde{\Delta})$. In this section we construct a graded algebra morphism $\varphi\colon \mathscr{B} \to \widetilde{\mathcal{O}}$ which will turn out to be a BV morphism when $\iota$ satisfies certain further conditions. In order to define such map, we will use an auxiliary space\footnote{Such a space, consisting of matrices with entries in $\widetilde{\mathcal O}$, appears in \cite{AGS}, which was our inspiration.}
\begin{equation*}
    \widetilde{\mathcal{O}} \otimes \underline{\operatorname{End}} {V}
    :=
    \widetilde{\mathcal{O}} \otimes \underline{\operatorname{End}}({\bigoplus}_{i\in \overline{Q}_0} V_i)
    \cong
    \widetilde{\mathcal{O}} \otimes {\bigoplus}_{i,j \in \overline{Q}_0} \underline{\operatorname{Hom}}(V_i,V_j).
\end{equation*}
This space is a graded algebra, and we take the multiplication in $\underline{\operatorname{End}}{V}$ to be totally reversed composition: that is, $f \cdot g := (-1)^{f \cdot g} g\circ f$, for homogeneous $f,g$. We do this for compatibility with paths, which we read from left to right. 
\begin{defn}
 Let us now introduce the following elements of $\widetilde{\mathcal{O}} \otimes \underline{\operatorname{End}} {V}$:
\begin{itemize}
    \item Let $M_{\iota} := 1_{\widetilde{\mathcal{O}}} \otimes \sum_i \iota^i$ and $M_{\iota^{-1}} := 1_{\widetilde{\mathcal{O}}} \otimes \sum_i (\iota^i)^{-1}$. $M_{\iota}$ has degree $|\iota| = -p-1$, and $M_{\iota^{-1}}$ degree $p+1$.
    
    \item For each vertex $i \in \overline{Q}_0$, define $M_{e_i}$ to be $1_{\widetilde{\mathcal{O}}} \otimes \mathrm{id}_{V_i}$, which has degree 0.
    
    \item For each arrow $a \in \overline{Q}_1$, $M_a$ is defined as the image of $s^{a} \otimes \mathrm{id}_{\mathrm{H}^a} \in S^{a} \otimes \underline{\operatorname{End}}(\mathrm{H}^a)$ under
    \begin{equation}\label{eqn: M_a defn}
        S^{a} \otimes \underline{\operatorname{End}}(\mathrm{H}^a) \stackrel{\sim}{\to} S^{a} \otimes (\mathrm{H}^a)^* \otimes (\mathrm{H}^a) \hookrightarrow \mathcal{O} \otimes \underline{\operatorname{End}} {V}
        \twoheadrightarrow \widetilde{\mathcal{O}} \otimes \underline{\operatorname{End}} {V}.
    \end{equation}
    The first map in \eqref{eqn: M_a defn} uses the inverse to the isomorphism 
    \begin{equation*}
        (\mathrm{H}^a)^* \otimes (\mathrm{H}^a) \stackrel{\sim}{\to} \underline{\operatorname{End}}(\mathrm{H}^a); \quad x \otimes h \mapsto \left[g \mapsto (-1)^{x \cdot h} x(g)h \right],
    \end{equation*} and, in the second map in \eqref{eqn: M_a defn}, $S^{a} \otimes (\mathrm{H}^a)^* \cong (\mathrm{H}^a \otimes S^{-a})^*$ is included into $\mathcal{O}$ linearly and the arrow ``$a$'' labeling $\mathrm{H}^a = \underline{\operatorname{Hom}}(V_{s(a),}V_{t(a)})$ is forgotten upon inclusion into $\underline{\operatorname{End}} {V}$. As these are all degree 0 maps, $M_a$ has degree $|a|$. In coordinates we have $M_a = (-1)^{\mu + \nu} \tensor{(X^a)}{_{\mu}^{\nu}} \otimes \tensor{(h^{s(a) \to t(a)})}{^{\mu}_{\nu}}.$
\end{itemize}
If $a_1 \cdots a_n$ is a path in $\overline{Q}$, then we also define $M_{a_1 \cdots a_n} := M_{a_1} \cdots M_{a_n}$, which makes it clear why we chose reversed composition as the multiplication in $\underline{\operatorname{End}}{V}$. \end{defn}
The upshot of this is that we can view ``$M$'' as a map that assigns an $\widetilde{\mathcal{O}}$-valued matrix to every path in $\overline{Q}$. Better yet, path multiplication is also respected as additionally we have $M_{e_{s(a)}} \cdot M_a = M_a = M_a \cdot M_{e_{t(a)}}$, for example.

\begin{prop}\label{prop: commutativity}
    \textit{The matrices $M_{e_i}$ and $M_a$ (graded) commute with $M_{\iota}$, for all vertices $i \in \overline{Q}_0$ and arrows $a \in \overline{Q}_1$.}
\end{prop}
\begin{proof}
    It is clear that $[M_\iota, M_{e_i}] = 0$. We now consider $[M_{\iota}, M_a]$. Before proceeding, notice that
    \begin{equation*}
        \tensor{(h^{i \to j})}{^{\mu}_{\nu}} \cdot \tensor{(h^{k \to l})}{^{\rho}_{\sigma}}
        =
        (-1)^{(\mu + \nu)(\rho + \sigma)}\delta^{jk}\delta_{\nu}^{\rho}
        \tensor{(h^{i \to l})}{^{\mu}_{\sigma}}.
    \end{equation*}
    We then calculate
    \begin{align*}
        M_{\iota} \cdot M_a &= \left[{\sum}_i\tensor{(\iota^i)}{_{\mu}^{\rho}} \otimes \tensor{(h^{i \to i})}{^{\mu}_{\rho}}\right] \cdot
        \left[ (-1)^{\sigma + \nu} \tensor{(X^a)}{_{\sigma}^{\nu}} \otimes \tensor{(h^{s(a) \to t(a)})}{^{\sigma}_{\nu}} \right]\\
        &= (-1)^{\sigma + \nu + (\mu + \rho)(a + \sigma + \nu)}
        {\sum}_i \tensor{(\iota^i)}{_{\mu}^{\rho}}\tensor{(X^a)}{_{\sigma}^{\nu}} \otimes
        \underbrace{
        \tensor{(h^{i \to i})}{^{\mu}_{\rho}} \cdot \tensor{(h^{s(a) \to t(a)})}{^{\sigma}_{\nu}}}
        _{(-1)^{(\mu + \rho)(\sigma + \nu)}\delta^{i,s(a)}\delta_{\rho}^{\sigma}\tensor{(h^{i \to t(a)})}{^{\mu}_{\nu}}}\\
        &= (-1)^{\rho + \nu + a(\mu + \rho)}  \tensor{(\iota^{s(a)})}{_{\mu}^{\rho}}\tensor{(X^a)}{_{\rho}^{\nu}} \otimes \tensor{(h^{s(a) \to t(a)})}{^{\mu}_{\nu}}.
    \end{align*}
    Using that $\rho + a(\mu + \rho) = \mu + \iota + a\iota = \mu + \iota \mod 2$ (or $\tensor{(\iota^i)}{_{\mu}^{\rho}} = 0$), we get
    \begin{equation*}
        M_{\iota} \cdot M_a = (-1)^{\mu + \nu + \iota}  \tensor{(\iota^{s(a)})}{_{\mu}^{\rho}}\tensor{(X^a)}{_{\rho}^{\nu}} \otimes \tensor{(h^{s(a) \to t(a)})}{^{\mu}_{\nu}}.
    \end{equation*}
    Similarly calculating $M_a \cdot M_{\iota}$, we find that
    \begin{equation} \label{eq:Ymunu}
        [M_{\iota},M_{a}] = \underbrace{\left[
        (-1)^{\mu + \nu + \iota}  \tensor{(\iota^{s(a)})}{_{\mu}^{\rho}}\tensor{(X^a)}{_{\rho}^{\nu}}
        -
        (-1)^{p(\mu + \nu)}\tensor{(X^a)}{_\mu^\rho}\tensor{(\iota^{t(a)})}{_\rho^\nu}
        \right]}
        _{=: Y_{\mu}^{\nu}}
        \otimes \tensor{(h^{s(a) \to t(a)})}{^{\mu}_{\nu}}.
    \end{equation}
    By Remark \ref{rem:subandquot}, $Y_{\mu}^\nu = 0$ in $\widetilde{\mathcal{O}}$ precisely if it annihilates $\mathrm{H}^a_{\iota}$. To this end, let $f \in \mathrm{H}^a_\iota$. Then
    \begin{align*}
        Y_\mu^\nu(f) &= (-1)^{\mu + \nu + \iota}  \tensor{(\iota^{s(a)})}{_{\mu}^{\rho}}\tensor{f}{_{\rho}^{\nu}}
        -
        (-1)^{p(\mu + \nu)}\tensor{f}{_\mu^\rho}\tensor{(\iota^{t(a)})}{_\rho^\nu}\\
        &= (-1)^{f}  \tensor{(\iota^{s(a)})}{_{\mu}^{\rho}}\tensor{f}{_{\rho}^{\nu}}
        -
        (-1)^{p(f + \iota)}\tensor{f}{_\mu^\rho}\tensor{(\iota^{t(a)})}{_\rho^\nu}\\
        &= (-1)^f \tensor{\left(f \circ \iota - (-1)^{\iota f} \iota \circ f\right)}{_\mu^\nu} = 0.
    \end{align*}
    Thus $[M_\iota,M_a] = 0$.
\end{proof}
Now consider the degree 0 map $(1 \otimes \operatorname{str}):\widetilde{\mathcal{O}} \otimes \underline{\operatorname{End}} {V} \to \widetilde{\mathcal{O}}$. Due to cyclicity of $\operatorname{str}$ (i.e. $\operatorname{str}(f \cdot g) = (-1)^{fg}\operatorname{str}(g \cdot f)$) and commutativity of $\widetilde{\mathcal{O}}$, it is easy to see that $(1 \otimes \operatorname{str})$ is cyclically symmetric too. This observation, together with Proposition \ref{prop: commutativity}, allows us to define the map $\varphi:\mathscr{B} \to \widetilde{\mathcal{O}}$ as follows:

\begin{defn}
    Define $\varphi:\mathscr{B} \to \widetilde{\mathcal{O}}$ to be the graded algebra morphism obtained as the extension of the degree $0$ linear map $\mathscr{A}[p+1] = S^{\iota} \otimes \mathscr{A} \to \widetilde{\mathcal{O}}$ that sends
    \begin{equation*}
         s^{\iota} \otimes a_1\cdots a_n \mapsto (1 \otimes \operatorname{str})(M_{\iota} \cdot M_{a_1} \cdots M_{a_n}) \quad \text{and} \quad s^{\iota} \otimes e_{i} \mapsto (1 \otimes \operatorname{str})(M_{\iota} \cdot M_{e_i}),
    \end{equation*}
    for all non-constant paths $a_1 \cdots a_n$ and constant paths $e_i$.
\end{defn}
Notice that we can understand $\varphi$ diagrammatically, for example by writing $\varphi(s^{\iota} \otimes a_1\cdots a_n)$ as 
\begin{equation*}
    \vcenter{\hbox{\includegraphics[scale=1.2]{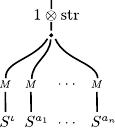}}}\;\;.
\end{equation*}
  Now, to see that $\varphi$ is well-defined, namely that it is compatible with the relations \eqref{eq:cyclicpropexplicit}, observe that 
\begin{align*}
    \varphi(s^{\iota} \otimes a_1\cdots a_n) 
    &= (1 \otimes \operatorname{str})(M_{\iota} \cdot M_{a_1} \cdots M_{a_n}) 
    = (-1)^{(\iota + a_1) \cdot (a_2 + \cdots + a_n)}(1 \otimes \operatorname{str})(M_{a_2} \cdots M_{a_n} \cdot M_{\iota} \cdot M_{a_1})\\
    &= (-1)^{a_1 \cdot (a_2 + \cdots + a_n)}(1 \otimes \operatorname{str})(M_{\iota} \cdot M_{a_2} \cdots M_{a_1}) = (-1)^{a_1 \cdot (a_2 + \cdots + a_n)} \varphi(s^{\iota} \otimes a_2\cdots a_1),
\end{align*}
showing that $\varphi$ respects the relations on $\mathscr{A}$.

The following proposition gives a coordinate description of this map.
\begin{prop}
    For constant paths $e_i$ we have
    \begin{equation*}
        \varphi(s^\iota \otimes e_i) = \mathrm{str}(\iota^i)1_{\widetilde{\mathcal{O}}} = (-1)^\mu\tensor{(\iota^i)}{_\mu^\mu}1_{\widetilde{\mathcal{O}}},
    \end{equation*} 
    while for non-constant closed paths $A = a_1 \cdots a_n$ we have
    \begin{equation*}
        \varphi(s^\iota \otimes A) = (-1)^{\iota + \mu_0(1+A) + \sum_{i=1}^n\mu_i a_i} \tensor{(\iota^{s(a_1)})}{_{\mu_0}^{\mu_1}}\tensor{(X^{a_1})}{_{\mu_1}^{\mu_2}} \cdots \tensor{(X^{a_n})}{_{\mu_n}^{\mu_0}}.
    \end{equation*}
\end{prop}
\begin{proof}
    Given the definition of $\varphi$, this result follows as a straightforward calculation.
\end{proof}

\subsection[Compatibility of the trace map with BV operators]{Compatibility of the trace map \texorpdfstring{$\varphi$}{phi} with BV operators}
We will now provide sufficient conditions for $\iota$ under which the map $\varphi$ intertwines the BV operators. As is common with extension to BV operators, the compatibility of $\varphi$ with the bracket (i.e. $\varphi$ being a morphism of Gerstenhaber/shifted Poisson algebras) requires no conditions (Proposition \ref{prop: gerstenhaber}). However, to get a  compatibility with the cobracket, i.e. with the full BV operator\footnote{An analogous situation appears in the construction of the BV action of the BF theory from a dgla. The classical master equation is satisfied automatically, while the dgla needs to be unimodular for the quantum master equation to hold \cite[Sec.~4.1]{MnevDiscreteBF}. Similarly, in \cite{ANPS}, there is a natural quasi-Gerstenhaber bracket on the moduli spaces of flat $G$-connections on a surface, for $G$ an odd metric Lie supergroup. To find a quasi-BV operator extending this bracket, the Lie algebra of $G$ needs to be unimodular, or one needs to choose a 1-dimensional foliation of the surface \cite[Eq.~(9)]{ANPS}. Also in \cite{ANPS}, the surface analogue of the BV map $\varphi$ from present work is compatible with the cobracket if the underlying associative algebra satisfies equation \cite[Eq.~(16)]{ANPS}.  
}, will require us to choose a correct normalization of $\iota^2$, in the $p=0$ case, and even a special form of $\iota$, in the $p=1$ case (Proposition \ref{prop: intertwining with (little) delta}, Theorem \ref{thm:BVrelaxedcondition}). Let us start with the following easy observation.
\begin{lem}\label{lem:IBLSymHOm}
Let $\mathcal{G}$ be a graded vector space and let $\widetilde{b}\colon \Sym^2 \mathcal{G} \to \mathcal G$ and $\widetilde{c} \colon \mathcal G \to \Sym^2 \mathcal G$ degree 1 maps such that their extensions to differential operators on  $\Sym \mathcal G$ define a BV operator $\Delta = \widetilde{b} + \widetilde{c}$. Let $(\widetilde O, \widetilde \Delta)$ be an arbitrary BV algebra. Then a map of graded algebras $\varphi \colon \Sym \mathcal G \to \widetilde O$ intertwines the BV operators if and only if 
\begin{equation}\label{eq:brPois}
    \varphi(\widetilde{b}(xy)) = \{\varphi(x),  \varphi(y) \}_{\widetilde{\Delta}}\quad \forall x, y \in \mathcal G,
\end{equation}
and 
\begin{equation}\label{eq:cobrDelta}
    \varphi (\widetilde{c}(x)) = \widetilde {\Delta} (\varphi (x)) \quad \forall x \in \mathcal G.
\end{equation}
Here, the bracket $\{a,b\}_{\widetilde{\Delta}} = \widetilde{\Delta}(ab) - \widetilde{\Delta} (a)b - (-1)^{a \cdot b} \widetilde{\Delta} (b) a$ is the odd Poisson bracket associated to $\widetilde{\Delta}$.
\end{lem}
\begin{proof}
    The result follows immediately from the following property of any BV operator $\Delta$:
    \[ \Delta (x_1 \dots x_n) = \sum_i \pm \Delta{(x_i)}\, x_1 \cdots \widehat{x}_i \cdots x_n + \sum_{i<j} \pm \{x_i, x_j\}_{\Delta} \,x_1 \cdots \widehat{x}_i \cdots \widehat{x}_j \cdots x_n, \]
    with obvious Koszul signs and hats denoting omission. 
\end{proof}
In our case we have $\mathcal{G} = \mathscr{A}[p+1]$ with $\widetilde{b}=\widetilde{\operatorname{br}}$ and $\widetilde{c} = \hbar\widetilde{\delta}$. We consider the two conditions  \eqref{eq:brPois} and \eqref{eq:cobrDelta} in Propositions \ref{prop: gerstenhaber} and \ref{prop: intertwining with (little) delta}, respectively. We will see that \eqref{eq:brPois} always holds, whereas \eqref{eq:cobrDelta} requires further assumptions on $\iota$. 
\begin{prop}\label{prop: gerstenhaber}
    \textit{The graded algebra morphism $\varphi\colon\mathscr{B} \to \widetilde{\mathcal{O}}$ satisfies, for all $x,y \in \mathscr{A}[p+1]$,
    \begin{equation}\label{eqn: bracket prf 0}
        \varphi\left(\widetilde{\operatorname{br}}(x \cdot y)\right)
        = \{\varphi(x), \varphi(y) \}_{\widetilde{\Delta}}.
    \end{equation}}
\end{prop}
Before proving this, we need another result:
\begin{lem}\label{lem: intertwining 1}
    \textit{Let $a$ and $b$ be arrows in $\overline{Q}_1$. Then
    \begin{equation}\label{eqn: intertwining prop 1}
        \vcenter{\hbox{\includegraphics[scale=1.2]{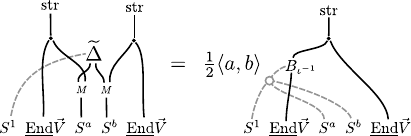}}}\;\;,
    \end{equation}
    where ``$\bullet$'' denotes multiplication in the relevant algebra (here, totally reversed composition in $\underline{\operatorname{End}}{V}$), and $B_{\iota^{-1}}(f) := \iota^{-1} \circ f + (-1)^{f \iota^{-1} + \iota^{-1}} f \circ \iota^{-1}$.}
\end{lem}
\begin{proof}
    The left hand side of \eqref{eqn: intertwining prop 1} sends (homogeneous) $s^1 \otimes \phi \otimes s^a \otimes s^b \otimes \psi$ to
    \begin{equation} \label{eqn: 001}
        (-1)^{\varepsilon + \mu + \nu + \rho + \sigma} 
        \operatorname{str}(\phi \cdot \tensor{(h^{s(a) \to t(a)})}{^{\mu}_{\nu}}) \;
        \widetilde{\Delta}(s^1 \otimes \tensor{(X^a)}{_{\mu}^{\nu}} \cdot \tensor{(X^b)}{_{\rho}^{\sigma}}) \;
        \operatorname{str}(\tensor{(h^{s(b) \to t(b)})}{^{\rho}_{\sigma}} \cdot \psi),
    \end{equation}
    where $\varepsilon$ is the Koszul sign resulting from overlaps in the diagram. Explicitly, $$\varepsilon = \phi + (\mu + \nu) + (\mu + \nu)(\mu + \nu + a) = \phi + a(\mu + \nu) \mod 2.$$
    Writing $\phi = \tensor{\phi}{_{\alpha}^{\beta}} \tensor{h}{^{\alpha}_{\beta}}$, we find that $\operatorname{str}(\phi \cdot \tensor{(h^{s(a) \to t(a)})}{^{\mu}_{\nu}}) = (-1)^{\mu} \tensor{\phi}{_{\nu}^{\mu}}$. Similarly, we find $\operatorname{str}(\tensor{(h^{s(b) \to t(b)})}{^{\rho}_{\sigma}} \cdot \psi) = (-1)^{\sigma}\tensor{\psi}{_{\sigma}^{\rho}}$. With this, \eqref{eqn: 001} becomes
    \begin{equation*}
        (-1)^{\phi + \nu + \rho + a(\mu + \nu)}\tensor{\phi}{_{\nu}^{\mu}}\tensor{\psi}{_{\sigma}^{\rho}} \langle a,b \rangle (-1)^{\iota + (\mu + \nu)(1 + b)}t_{\mu\rho}^{\nu\sigma}
        =
        (-1)^{\iota + \mu + \rho + (\mu + \nu) \iota} \langle a,b \rangle \tensor{\phi}{_{\nu}^{\mu}}\tensor{\psi}{_{\sigma}^{\rho}} t_{\mu\rho}^{\nu\sigma},
    \end{equation*}
    where we used that $|a| + |b| = p$ (or $\langle a,b \rangle = 0$). Using \eqref{eqn: t-components generalisation}, this is found to be
    \begin{align*}
        &\frac{1}{2} \langle a,b \rangle
        \left[
        (-1)^{\iota + \rho + (\mu + \nu) \iota}
        \tensor{\phi}{_{\nu}^{\mu}}\tensor{\psi}{_{\sigma}^{\rho}}
        \tensor{(\iota^{-1})}{_{\mu}^{\sigma}}\delta_{\rho}^{\nu}
        +
        (-1)^{\rho}
        \tensor{\phi}{_{\nu}^{\mu}}\tensor{\psi}{_{\sigma}^{\rho}}
        \delta_{\mu}^{\sigma}\tensor{(\iota^{-1})}{_{\rho}^{\nu}}
        \right]\\
        &= \frac{1}{2}\langle a,b \rangle
        \left[
        (-1)^{\phi \iota + \iota}\operatorname{str}(\psi \circ \iota^{-1} \circ \phi)
        + \operatorname{str}(\psi \circ \phi \circ \iota^{-1})
        \right]\\
        &=
        \frac{1}{2}\langle a,b \rangle
        (-1)^{\phi p} \operatorname{str}\left((\iota^{-1} \circ \phi + (-1)^{\phi \iota + \iota} \phi \circ \iota^{-1}) \circ \psi\right)\\
        &=
        \frac{1}{2}\langle a,b \rangle (-1)^{\phi p} 
        \operatorname{str}(B_{\iota^{-1}}(\phi) \cdot \psi),
    \end{align*}
    the right hand side of \eqref{eqn: 001}. Note that it helps to use ``$\operatorname{str}(f) \neq 0 \implies |f| = 0 \mod{2}$'' in the steps above.
\end{proof}
\begin{lem}\label{lem: intertwining 2}
    \textit{Define $B_{M_{\iota}}\colon \widetilde{\mathcal{O}} \otimes \underline{\operatorname{End}}{V} \to \widetilde{\mathcal{O}} \otimes \underline{\operatorname{End}}{V}$ by
    $B_{M_{\iota}}(\Phi) = M_{\iota} \cdot \Phi + (-1)^{\iota \Phi + \iota} \Phi \cdot M_{\iota}$. Then
    \begin{equation*}
        \vcenter{\hbox{\includegraphics[scale=1.2]{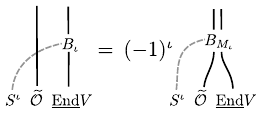}}}\;\;.
    \end{equation*}}
\end{lem}
\begin{proof}
    The right hand side sends (homogeneous) $f \otimes \phi$ to
    \begin{align*}
        B_{M_{\iota}}(f \otimes \phi)
        &=
        (1 \otimes \iota) \cdot (f \otimes \phi) + (-1)^{\iota(f + \phi) + \iota}(f \otimes \phi) \cdot (1 \otimes \iota)\\
        &= (-1)^{\iota f + \iota}f \otimes (\iota \circ \phi + (-1)^{\iota \phi + \iota} \phi \circ \iota) 
        = (-1)^{\iota f + \iota}f \otimes B_{\iota}(\phi)\\
        &= (-1)^{\iota} (1 \otimes B_{\iota}) (f \otimes \phi),
    \end{align*}
    which is precisely what the left hand side sends $f \otimes \phi$ to.
\end{proof}
Combining Lemmas \ref{lem: intertwining 1} and \ref{lem: intertwining 2} (and using the obvious variant with $\iota^{-1}$ instead of $\iota$), we have (for any arrows $a,b \in \overline{Q}_1$)
\begin{equation}\label{eqn: intertwining 3}
    \vcenter{\hbox{\includegraphics[scale=1.2]{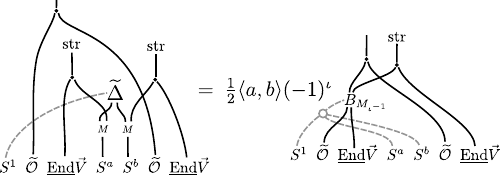}}}\;\;.
\end{equation}
Introducing the notation
\begin{equation*}
    \vcenter{\hbox{\includegraphics[scale=1.2]{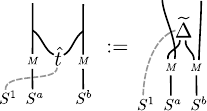}}}\;\;,
\end{equation*}
we can write \eqref{eqn: intertwining 3} elegantly as
\begin{equation}\label{eqn: intertwining 4}
    \vcenter{\hbox{\includegraphics[scale=1.2]{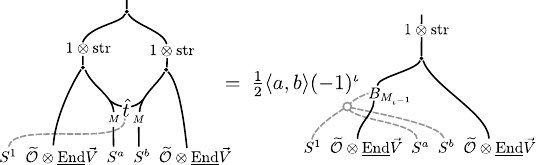}}}\;\;.
\end{equation}
\begin{proof}[Proof of Proposition \ref{prop: gerstenhaber}]
It suffices to show that \eqref{eqn: bracket prf 0} holds for all $x,y$ of the forms $s^{\iota} \otimes P_1$ and $s^{\iota} \otimes P_2$ where $P_1$ and $P_2$ are paths in $\overline{Q}$ (thought of in $\mathscr{A}$).
If $P_1$ or $P_2$ is a constant path, then it is clear that both sides of \eqref{eqn: bracket prf 0} vanish. It therefore remains to prove \eqref{eqn: bracket prf 0} on non-constant paths.

Let $A$ and $B$ be non-constant paths. We want to show that 
\begin{equation}\label{eqn: br intertwining identity}
    \varphi\left(\widetilde{\operatorname{br}}(s^1 \otimes (s^{\iota} \otimes A_k) \cdot (s^\iota \otimes B_l))\right)
    = \{ s^1 \otimes \varphi(s^{\iota} \otimes A_k) \otimes \varphi(s^\iota \otimes B_l)\}_{\widetilde{\Delta}},
\end{equation}
where we have now explicitly included the degree-correcting $s$ factors. 

The $ij$'th term of $\widetilde{\operatorname{br}}(s^1 \otimes (s^{\iota} \otimes  A_k) \cdot (s^{\iota} \otimes B_l))$ can be written diagrammatically as
\begin{equation*}
    \includegraphics[scale=1.2]{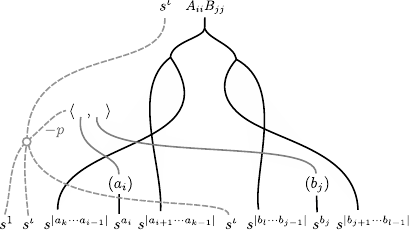}\;\;,
\end{equation*}
where we let $s^{|a_k \cdots a_{i-1}|} = s^{e_{t(a_{k-1})}}$ if $i = k \mod \text{``the length of $A$''}$, for example.
Precomposing with the permutation
\begin{equation}\label{eqn: permutation 1}
    \vcenter{\hbox{\includegraphics[scale=1.2]{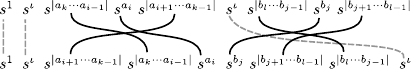}}}\;\;,
\end{equation}
this becomes
\begin{equation}\label{eqn: ijth term (permuted)}
    \vcenter{\hbox{\includegraphics[scale=1.2]{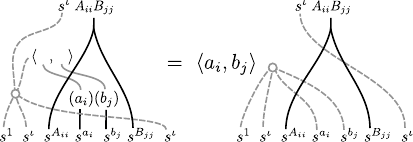}}}\;\;,
\end{equation}
where to see the last equality it helps to recall that the degree of $\iota$ has opposite parity to the degree of any element of $\mathscr{A}$. Applying $\varphi$ to \eqref{eqn: ijth term (permuted)} yields 
\begin{equation}\label{eqn: ijth term (permuted + trace applied)}
    \vcenter{\hbox{\includegraphics[scale=1.2]{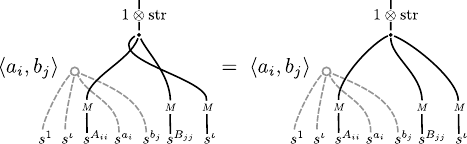}}}\;\;,
\end{equation}
where we have used cyclicity of $1 \otimes \operatorname{str}$. By associativity of composition and linearity of $1 \otimes \operatorname{str}$, \eqref{eqn: ijth term (permuted + trace applied)} is nothing but the $ij$'th term of $\varphi(\widetilde{\operatorname{br}}(s^1 \otimes (s^{\iota} \otimes  A_k) \cdot (s^{\iota} \otimes B_l)))$ precomposed by the permutation \eqref{eqn: permutation 1}.

Consider now the $ij$'th term of the right hand side of \eqref{eqn: br intertwining identity}, which is given by
\begin{equation} \label{eqn: ijth term num.2}
    \vcenter{\hbox{\includegraphics[scale=1.2]{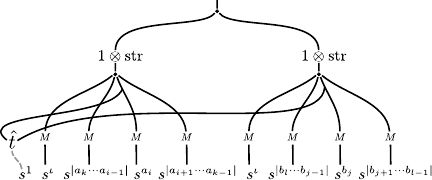}}}\;\;.
\end{equation}
Note that we could move the action of $\hat{t}$ (which acts on $\widetilde{\mathcal{O}}$) past $1 \otimes \operatorname{str}$ (which acts on $\underline{\operatorname{End}}{V}$). Now precompose \eqref{eqn: ijth term num.2} with the permutation \eqref{eqn: permutation 1} to get
\begin{equation*}
    \vcenter{\hbox{\includegraphics[scale=1.2]{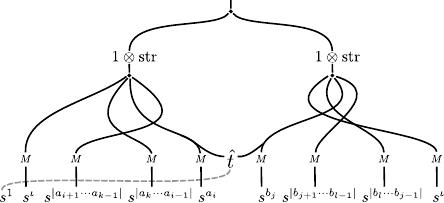}}}\;\;.
\end{equation*}
Using cyclicity of $1 \otimes \operatorname{str}$ and Proposition \ref{prop: commutativity}, this is nothing but
\begin{equation} \label{eqn: intertwining 5}
    \vcenter{\hbox{\includegraphics[scale=1.2]{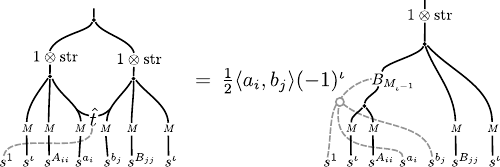}}}\;\;,
\end{equation}
where the equality follows from identity \eqref{eqn: intertwining 4}. Now notice that
\begin{equation*}
    B_{M_{\iota^{-1}}}(M_{\iota} \cdot M_{A_{ii}})
    =
    M_{\iota^{-1}} \cdot M_{\iota} \cdot M_{A_{ii}} + (-1)^{\iota(\iota + A_{ii}) + \iota} M_{\iota} \cdot M_{A_{ii}} \cdot M_{\iota^{-1}} = (-1)^\iota 2 M_{A_{ii}},
\end{equation*}
so that \eqref{eqn: intertwining 5} is equal to
\begin{equation*}
    \vcenter{\hbox{\includegraphics[scale=1.2]{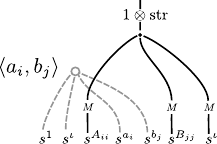}}}.
\end{equation*}
Hence, the $ij$'th term of the right hand side of \eqref{eqn: br intertwining identity} precomposed by the permutation \eqref{eqn: permutation 1} is precisely the $ij$'th term of the left hand side of \eqref{eqn: br intertwining identity} precomposed by the permutation \eqref{eqn: permutation 1}. In other words, we have have shown that \eqref{eqn: br intertwining identity} holds, completing the proof.
\end{proof}

Let us now turn to the compatibility of the cobracket and the BV operator, i.e.\ condition \eqref{eq:cobrDelta}. Recall that we had two nonzero complex constants: $\hbar$ in the definition of necklace BV operator $\Delta_\hbar$ (Definition \ref{def:necklaceBV}), and $\lambda$ given by $\iota^2 = \lambda\id$. These need to be related in order to match the rescaled necklace cobracket $\hbar \widetilde{\delta}$ and the BV operator $\widetilde{\Delta}$ on the representation variety. Moreover, in the $p=1$ case, we will moreover assume that $\iota$ is proportional to identity. We will explain how this can be relaxed later.
\begin{prop}\label{prop: intertwining with (little) delta}
    \textit{
    Assume,
    \begin{itemize}
        \item if $p=0$, that $\lambda  = -\frac{1}{2\hbar}$;
        \item if $p=1$, that $\iota^i = \mu \mathrm{id}_{V_i}$ for all $i$, where $\mu$ is a root of $\frac{1}{\hbar}$ (necessarily, $\lambda = \frac{1}{\hbar}$). 
    \end{itemize}
    Then, for all $x \in \mathscr{A}[p+1]$,
    \begin{equation}\label{eqn: intertwining with (little) delta identity}
        \varphi(\hbar \widetilde{\delta}(x)) = \widetilde{\Delta}\left(\varphi(x)\right).
    \end{equation}}
\end{prop}

Once again, we need another result first:

\begin{lem}\label{lem: the WORM}
    \textit{Let $a$ and $b$ be arrows in $\overline{Q}_1$. Then
    \begin{equation}\label{eqn: cobracket prf 1}
        \vcenter{\hbox{\includegraphics[scale=1.2]{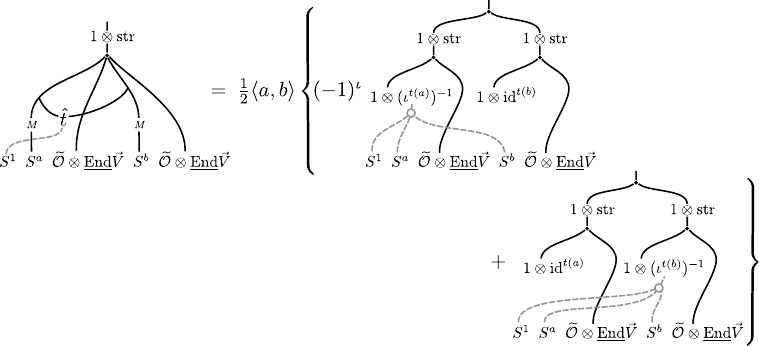}}}.
    \end{equation}}
\end{lem}
\begin{proof}
    Given the definition of $\hat{t}$ in diagrams, the left hand side of \eqref{eqn: cobracket prf 1} is understood to be
    \begin{equation*}
        \vcenter{\hbox{\includegraphics[scale=1.2]{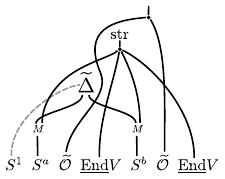}}}\;\;.
    \end{equation*}
    Letting $\Phi := f \otimes \phi$ and $\Psi := g \otimes \psi$ in $\widetilde{\mathcal{O}} \otimes \underline{\operatorname{End}}{V}$, this sends $s^1 \otimes s^a \otimes \Phi \otimes s^b \otimes \Psi$ to
    \begin{equation}\label{eqn: cobracket prf 2}
        (-1)^{\varepsilon + \mu + \nu + \rho + \sigma}
        \widetilde{\Delta}\left( s^1 \otimes \tensor{(X^a)}{_\mu^\nu} \cdot \tensor{(X^b)}{_\rho^\sigma}\right)
        \operatorname{str}\left(\tensor{(h^{s(a) \to t(a)})}{^\mu_\nu} \cdot \phi \cdot \tensor{(h^{s(b) \to t(b)})}{^\rho_\sigma} \cdot \psi\right)
        f \cdot g,
    \end{equation}
    where $\varepsilon$ is the Koszul sign arising from overlaps in the diagram. That is,
    \begin{align*}
        \varepsilon &= (\mu + \nu + a)(\mu + \nu) + (\mu + \nu) + \Phi(b + \rho + \sigma) + f(\mu + \nu) + g \psi\\
        &= (a + f)(\mu + \nu) + \Phi(b + \rho + \sigma) + g \psi \mod{2}.
    \end{align*}
    We then find that \begin{align*}
        \operatorname{str}\left(\tensor{(h^{s(a) \to t(a)})}{^\mu_\nu} \cdot \phi \cdot \tensor{(h^{s(b) \to t(b)})}{^\rho_\sigma} \cdot \psi\right)
        &=
        (-1)^{\varepsilon'}
        \operatorname{str}\left(
        \psi \circ \tensor{(h^{s(b) \to t(b)})}{^\rho_\sigma} \circ \phi \circ \tensor{(h^{s(a) \to t(a)})}{^\mu_\nu}
        \right)\\
        &= (-1)^{\varepsilon' + \mu}\tensor{\phi}{_\nu^\rho} \tensor{\psi}{_\sigma^\mu},
    \end{align*}
    where 
    \begin{equation*}
        \varepsilon' = (\mu + \nu)(\phi + \rho + \sigma + \psi) + \phi(\rho + \sigma + \psi) + (\rho + \sigma)\psi \mod{2}.
    \end{equation*}
    Using this as well as \eqref{eqn: delta-tilde, explicitely}, we get that \eqref{eqn: cobracket prf 2} is given by
    \begin{equation}\label{eqn: cobracket prf 3}
        \langle a,b \rangle(-1)^{\varepsilon''}t_{\mu\rho}^{\nu\sigma} \tensor{\phi}{_\nu^\rho} \tensor{\psi}{_\sigma^\mu} f \cdot g,
    \end{equation}
    where
    \begin{equation*}
        \varepsilon'' = \varepsilon + \mu + \nu + \rho + \sigma + \varepsilon' + \mu + \iota + (\mu + \nu)(1 + b) \mod{2}.
    \end{equation*}
    By using
    \begin{itemize}
        \item ``$\mu + \nu + \rho + \sigma = \iota \mod{2}$'' or ``$t_{\mu\rho}^{\nu\sigma} = 0$'',
        \item ``$a + b = p$'' or ``$\langle a,b \rangle = 0$'', and
        \item ``$\phi = \nu + \rho \mod{2}$ and $\psi = \sigma + \mu \mod{2}$'' or ``$\tensor{\phi}{_\nu^\rho} \tensor{\psi}{_\sigma^\mu} = 0$'',
    \end{itemize}
    a calculation shows that $\varepsilon''$ in \eqref{eqn: cobracket prf 3} can be taken to be
    \begin{equation*}
        \varepsilon'' = \iota + (\mu + \nu)(1 + \phi) + \Phi(b + \iota) + g\psi + \phi\psi + \mu.
    \end{equation*}
    Now, by \eqref{eqn: t-components generalisation}, \eqref{eqn: cobracket prf 3} is found to be
    \begin{equation}\label{eqn: cobracket prf 4}
        \frac{1}{2}\langle a,b \rangle
        (-1)^{\iota + (\mu + \nu)(1 + \phi) + \Phi(b + \iota) + g\psi + \phi \psi}
        \left[
        \tensor{(\iota^{-1})}{_{\mu}^{\sigma}} \delta_{\rho}^{\nu} + (-1)^{\iota \cdot (\mu + \nu + 1)} \delta_{\mu}^{\sigma}\tensor{(\iota^{-1})}{_{\rho}^{\nu}}
        \right] \tensor{\phi}{_\nu^\rho} \tensor{\psi}{_\sigma^\mu} f \cdot g.
    \end{equation}
    The first term in the square bracket of \eqref{eqn: cobracket prf 4} gives rise to the term
    \begin{align*}
        &\frac{1}{2}\langle a,b \rangle(-1)^{\mu + \nu + \Phi(b + \iota) + \iota g + \iota} 
        \tensor{(\iota^{-1})}{_{\mu}^{\sigma}} \tensor{\psi}{_\sigma^\mu} \delta_{\rho}^{\nu}  \tensor{\phi}{_\nu^\rho}  f \cdot g \\
        &=
        \frac{1}{2}\langle a,b \rangle(-1)^{\Phi(b + \iota) + \iota g + \iota} 
        \operatorname{str}\left(\psi \circ (\iota^{t(b)})^{-1}\right)
        \operatorname{str}\left((\mathrm{id}^{t(a)}) \circ \phi \right) f \cdot g\\
        &=\frac{1}{2}\langle a,b \rangle(-1)^{\Phi(b + \iota) + \iota g} 
        \operatorname{str}\left((\mathrm{id}^{t(a)}) \cdot \phi \right)
        \operatorname{str}\left((\iota^{t(b)})^{-1} \cdot \psi \right) f \cdot g,
    \end{align*}
    where we have $\delta_{\rho}^\nu = \tensor{(\mathrm{id}^{t(a)})}{_{\rho}^{\nu}}$ since $\nu$ labels the basis $\{v_{\nu}^{(t(a))}\}$ of $V_{t(a)}$, for example. Finally, this can be rewritten as 
    \begin{equation*}
        \frac{1}{2}\langle a,b \rangle (-1)^{\Phi(a + 1)}
        (1 \otimes \operatorname{str})\left((1 \otimes \mathrm{id}^{t(a)}) \cdot \Phi \right)
        (1 \otimes \operatorname{str})\left((1 \otimes (\iota^{t(b)})^{-1}) \cdot \Psi \right),
    \end{equation*}
    which is precisely what the second term on the right hand side of \eqref{eqn: cobracket prf 1} sends $s^1 \otimes s^a \otimes \Phi \otimes s^b \otimes \Psi$ to.

    Similarly, we find that the second term in the square bracket of \eqref{eqn: cobracket prf 4} gives rise to the term
    \begin{equation*}
        \frac{1}{2}\langle a,b \rangle (-1)^{\Phi b + \iota}
        (1 \otimes \operatorname{str})\left((1 \otimes (\iota^{t(a)})^{-1}) \cdot \Phi \right)
        (1 \otimes \operatorname{str})\left((1 \otimes \mathrm{id}^{t(b)}) \cdot \Psi \right),
    \end{equation*}
    which is precisely what the first term on the right hand side of \eqref{eqn: cobracket prf 1} sends $s^1 \otimes s^a \otimes \Phi \otimes s^b \otimes \Psi$ to.
\end{proof}

\begin{proof}[Proof of Proposition \ref{prop: intertwining with (little) delta}]
We now really do have to think purely in terms of $\mathbb{Z}_2$-gradings. First note that it suffices to prove the result for $x$ of the form $s^{\iota} \otimes P$ where $P$ is a path in $\overline{Q}$. If $P$ has length $\leq 1$, then it is clear that both sides of \eqref{eqn: intertwining with (little) delta identity} vanish. Now let $A$ be a path of length $\geq 2$. We show that
\begin{equation}\label{eqn: intertwining with (little) delta 2}
    \varphi\left( \hbar \widetilde{\delta}(s^1 \otimes (s^\iota \otimes A_k))\right) = \widetilde{\Delta}\left(s^1 \otimes \varphi(s^{\iota} \otimes A_k)\right),
\end{equation}
where we have made the $s$'s explicit. Recall the definition of $\delta(A_k)$ in \eqref{eqn: delta usual-notation expression}, which we can alternatively write as
\begin{equation*}
    \delta(A_k) 
    =
    \frac{1}{2} \sum_{i<j}^{A} (-1)^{\varepsilon_{ki}^A + A_{ij} \cdot a_j} \langle a_i, a_j \rangle \left( A_{ij} \otimes A_{ji} + (-1)^{p + 1 + A_{ij} \cdot A_{ji}} A_{ji} \otimes A_{ij}\right),
\end{equation*}
where here ``$i < j$'' means that $i$ and $j$ are distinct indices with $i$ appearing before $j$ as one travels along the path starting at $k$. As was explained under \eqref{eqn: shift of br and delta}, there is no sign picked up in the shift of $\delta$ to $\overline{\delta}$. So $\overline{\delta}(s^{\iota} \otimes A_k)$ is given by
\begin{equation*}
    \overline{\delta}(s^{\iota} \otimes A_k) 
    =
    \frac{1}{2} \sum_{i<j}^{A} (-1)^{\varepsilon_{ki}^A + A_{ij} \cdot a_j} \langle a_i, a_j \rangle \left( (s^{\iota} \otimes A_{ij}) \otimes (s^{\iota} \otimes A_{ji}) + (-1)^{\iota + A_{ij} \cdot A_{ji}} (s^{\iota} \otimes A_{ji}) \otimes (s^{\iota} \otimes A_{ij})\right).
\end{equation*}
Making the $s^1$ factor explicit, it follows that $\widetilde{\delta}(s^1 \otimes (s^{\iota} \otimes A_k))$ is
\begin{equation*}
    \widetilde{\delta}(s^1 \otimes (s^{\iota} \otimes A_k)) = \sum_{i<j}^{A} (-1)^{\varepsilon_{ki}^A + A_{ij} \cdot a_j} \langle a_i, a_j \rangle (s^{\iota} \otimes A_{ij})(s^{\iota} \otimes A_{ji}),
\end{equation*}
the $ij$'th term of which can be written diagrammatically as
\begin{equation}\label{eqn: cobracket prf 5}
    \vcenter{\hbox{\includegraphics[scale=1.2]{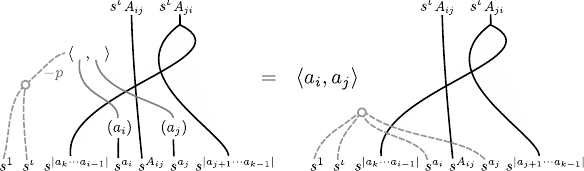}}}\;\;,
\end{equation}
where the appearance of the two $s^{\iota}$'s at the top can be reconciled due to their total degree being $2 \iota = 0 \mod{2}$. Precomposing \eqref{eqn: cobracket prf 5} with the permutation
\begin{equation}\label{eqn: permutation 2}
    \vcenter{\hbox{\includegraphics[scale=1.2]{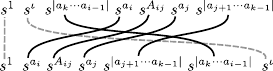}}}\;\;
\end{equation}
yields
\begin{equation*}
    \vcenter{\hbox{\includegraphics[scale=1.2]{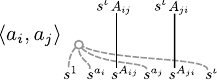}}}\;\;.
\end{equation*}
Applying the map $\varphi$ gives
\begin{equation*}
    \vcenter{\hbox{\includegraphics[scale=1.2]{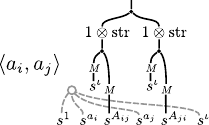}}}\;\;.
\end{equation*}
And so, the $ij$'th term of $\varphi(\hbar\widetilde{\delta}(s^1 \otimes (s^{\iota} \otimes A_k)))$ precomposed with the permutation \eqref{eqn: permutation 2} is given by
\begin{equation}\label{eqn: cobracket prf 7}
    \hbar \langle a_i,a_j \rangle (-1)^{A_{ij} \cdot a_j} (1 \otimes \operatorname{str})(M_{\iota} \cdot M_{A_{ij}}) \cdot (1 \otimes \operatorname{str})(M_{\iota} \cdot M_{A_{ji}}).
\end{equation}

We now consider the $ij$'th term of $\widetilde{\Delta}(s^1 \otimes \varphi(s^{\iota} \otimes A_k))$, which is given by
\begin{equation*}
    \vcenter{\hbox{\includegraphics[scale=1.2]{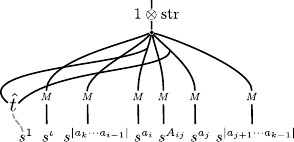}}}\;\;.
\end{equation*}
Precomposing this with the permutation \eqref{eqn: permutation 2}, we get
\begin{equation*}
    \vcenter{\hbox{\includegraphics[scale=1.2]{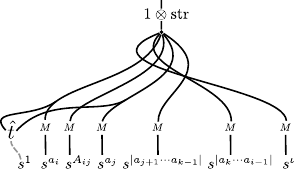}}}\;\;.
\end{equation*}
By cyclicity of $1 \otimes \operatorname{str}$ and Proposition \ref{prop: commutativity}, this is nothing but
\begin{equation*}
    \vcenter{\hbox{\includegraphics[scale=1.2]{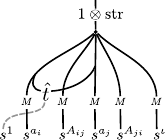}}}\;\;.
\end{equation*}
Applying Lemma \ref{lem: the WORM}, we see that the $ij$'th term of $\widetilde{\Delta}(s^1 \otimes \varphi(s^{\iota} \otimes A_k))$ precomposed by the permutation \eqref{eqn: permutation 2} is
\begin{align*}
    \frac{1}{2}\langle a_i, a_j \rangle (-1)^{A_{ij} \cdot a_j}
    \Big\{
     (-1)^{\iota}(1 \otimes \operatorname{str})\big(\underbrace{(1 \otimes (\iota^{t(a_i)})^{-1})}_{M_{\iota^{-1}} \cdot M_{e_{t(a_i)}}} \cdot M_{A_{ij}}\big) \cdot 
     (1 \otimes \operatorname{str})\big(\underbrace{(1 \otimes \mathrm{id}^{t(a_j)})}_{M_{e_{t(a_j)}}} \cdot M_{A_{ji}} \cdot M_{\iota} \big)\\
    +
    (1 \otimes \operatorname{str})\big(\underbrace{(1 \otimes \mathrm{id}^{t(a_i)})}_{M_{e_{t(a_i)}}} \cdot M_{A_{ij}}\big) \cdot 
    (1 \otimes \operatorname{str})\big(\underbrace{(1 \otimes (\iota^{t(a_j)})^{-1})}_{M_{\iota^{-1}} \cdot M_{e_{t(a_j)}}} \cdot M_{A_{ji}} \cdot M_{\iota} \big)
    \Big\},
\end{align*}
which, using Proposition \ref{prop: commutativity} as well as that $M_{e_{t(a_i)}} \cdot M_{A_{ij}} = M_{A_{ij}}$, for example, is equal to
\begin{align}\label{eqn: cobracket prf 6}
\begin{split}
    \frac{1}{2}\langle a_i, a_j \rangle (-1)^{A_{ij} \cdot a_j}
    \Big\{
     (-1)^{\iota}(1 \otimes \operatorname{str})\big(M_{\iota^{-1}} \cdot M_{A_{ij}}\big) \cdot 
     (1 \otimes \operatorname{str})\big( M_{\iota} \cdot M_{A_{ji}}\big)\\
    +
    (1 \otimes \operatorname{str})\big(M_{A_{ij}}\big) \cdot 
    (1 \otimes \operatorname{str})\big(M_{A_{ji}}\big)
    \Big\}.
\end{split}
\end{align}

For \eqref{eqn: intertwining with (little) delta 2} to hold, we need \eqref{eqn: cobracket prf 6} to be equal to \eqref{eqn: cobracket prf 7}. For this we consider two cases:
\begin{itemize}
    \item $p=0 \mod{2}$.\\
    Since $\iota^2$ is equal to $\lambda \mathrm{id}$, and since $M_\iota$ commutes with $M_A$ for any path $A$, we have
    \begin{align*}
        \lambda (1 \otimes \operatorname{str})(M_{A})
        &=
        (-1)^\iota(1 \otimes \operatorname{str})(M_{\iota} \cdot M_{\iota} \cdot M_A)
        =
        (-1)^{\iota + \iota(\iota + A)} (1 \otimes \operatorname{str})(M_{\iota} \cdot M_A \cdot M_{\iota})\\
        &=
        (1 \otimes \operatorname{str})(M_{\iota} \cdot M_{\iota} \cdot M_A)
        = (-1)^\iota\lambda (1 \otimes \operatorname{str})(M_{A})
        = -\lambda (1 \otimes \operatorname{str})(M_{A}),
    \end{align*}
    where we used that $\iota$ is odd. Thus $(1 \otimes \operatorname{str})(M_A)$ vanishes for all paths $A$. Furthermore, $\iota^{-1} = \frac{1}{\lambda} \iota$, and so \eqref{eqn: cobracket prf 6} reduces to
    \begin{equation*}
        -\frac{1}{2 \lambda}\langle a_i, a_j \rangle (-1)^{A_{ij} \cdot a_j}
         (1 \otimes \operatorname{str})\big(M_{\iota} \cdot M_{A_{ij}}\big) \cdot 
         (1 \otimes \operatorname{str})\big( M_{\iota} \cdot M_{A_{ji}}\big),
    \end{equation*}
    which agrees with \eqref{eqn: cobracket prf 7} when $\lambda = -(2 \hbar)^{-1}$.

    \item $p=1 \mod{2}$.\\
    Setting $\iota = (\lambda)^{1/2} \mathrm{id}$ (so that $\iota^2 = \lambda \mathrm{id}$), \eqref{eqn: cobracket prf 6} reduces to
    \begin{equation*}
        \frac{1}{\lambda}\langle a_i, a_j \rangle (-1)^{A_{ij} \cdot a_j}
         (1 \otimes \operatorname{str})\big(M_{\iota} \cdot M_{A_{ij}}\big) \cdot 
         (1 \otimes \operatorname{str})\big( M_{\iota} \cdot M_{A_{ji}}\big),
    \end{equation*}
    which agrees with \eqref{eqn: cobracket prf 7} when $\lambda = \hbar^{-1}$.
\end{itemize}
In either case, for suitably defined maps $\iota = \{\iota^i\}_{i \in \overline{Q}_0}$ we have that
\begin{equation*}
    \varphi(\hbar\widetilde{\delta}(s^1 \otimes (s^{\iota} \otimes A_k)))
    =
    \widetilde{\Delta}(s^1 \otimes \varphi(s^{\iota} \otimes A_k)).
\end{equation*}
\end{proof}

Combining Lemma \ref{lem:IBLSymHOm} and Propositions \ref{prop: gerstenhaber} and \ref{prop: intertwining with (little) delta}, we arrive at the following:
\begin{thm}\label{thm: BVmorph}
    \textit{If $p = 0 \mod{2}$, suppose that the odd maps $\{\iota^i\}_{i \in \overline{Q}_0}$ satisfy $(\iota^i)^2 = -(2 \hbar)^{-1}\mathrm{id}_{V_i}$ for each $i$. If instead $p=1 \mod{2}$, suppose that the even maps $\{\iota^i\}_{i \in \overline{Q}_0}$ are given by $\iota^i = (\hbar)^{-1/2} \mathrm{id}_{V_i}$ for each $i$, where the square root $\hbar^{-1/2}$ is independent of vertex $i$. Then the map $\varphi\colon \mathscr{B} \to \widetilde{\mathcal{O}}$ is a BV algebra morphism. That is, $\varphi$ intertwines the BV operators:
    \begin{equation}\label{eqn: BV operator intertwining}
        \varphi \circ \Delta_\hbar = \widetilde{\Delta} \circ \varphi.
    \end{equation}}
\end{thm}

\subsection{Relaxing assumptions in the \texorpdfstring{$p=1$}{p=1} case}
The $p=1$ case of Theorem \ref{thm: BVmorph} is somewhat unsatisfying: the assumptions involved on $\iota$ are strong, fixing it up to a sign. In particular, commuting with $\iota$ becomes an empty condition, so that $\widetilde{\mathcal{O}} = \mathcal{O}$. In this subsection we explain how one can relax this assumption. We will still assume normalisation according to choosing $\lambda = \frac 1 \hbar$, so that $\iota^2 = \frac{1}{\hbar} \mathrm{id}$. We work purely in the $p=1$ case.

Suppose we have a collection $\{V_i\}_{i \in \overline{Q}_0}$ of $\mathbb{Z}_2$-graded vector spaces and a collection $\{\iota^i:V_i \to V_i\}_{i \in \overline{Q}_0}$ of even automorphisms such that $(\iota^i)^2 = \frac 1 \hbar \mathrm{id}_{V_i}$. Let's denote this data by $(V,\iota)$. We have the associated BV algebra $(\widetilde{\mathcal{O}},\widetilde{\Delta})$ and graded algebra morphism $\varphi:\mathscr{B}=\mathrm{Sym}(\mathscr{A}) \to \widetilde{\mathcal{O}}$. Let us now fix an arbitrary root $\mu$ of $\frac{1}{\hbar}$. For each $i$ there is a graded vector space decomposition $V_i = V_{i,+} \oplus V_{i,-}$ relative to which $\iota^i$ is block diagonal, taking the form
\begin{equation*}
\iota^i
=
\mu
\begin{bmatrix}
    \mathrm{id}_{V_{i,+}} &  \\
     & -\mathrm{id}_{V_{i,-}}
\end{bmatrix}.
\end{equation*}
As such, $(V,\iota)$ decomposes into two collections of $\mathbb{Z}_2$-graded vector spaces and even automorphisms, namely
$$
(V_{+},\mu \mathrm{id}) = (\{V_{i,+}\},\{\mu \mathrm{id}_{V_{i,+}}\})\quad \text{and} \quad (V_{-},-\mu \mathrm{id}) = (\{V_{i,-}\},\{-\mu \mathrm{id}_{V_{i,-}}\}).
$$
These give rise to two corresponding BV algebras $(\mathcal{O}_\pm,\widetilde{\Delta}_\pm)$ and graded algebra morphisms $\varphi_\pm:\mathscr{B} \to \mathcal{O}_{\pm}$ which are in fact BV algebra morphisms by Theorem $\ref{thm: BVmorph}$.

Now, let $a \in \overline{Q}_1$ be an arrow, and consider the vector space $\mathrm{H}^a_{\iota}$. One sees that it decomposes as
\begin{equation*}
    \mathrm{H}^a_{\iota} := \{f \in \underline{\mathrm{Hom}}(V_{s(a)},V_{t(a)}) \; | \; [\iota,f] = 0\} = \underbrace{\underline{\mathrm{Hom}}(V_{s(a),+},V_{t(a),+})}_{\mathrm{H}^a_+} \oplus \underbrace{\underline{\mathrm{Hom}}(V_{s(a),-},V_{t(a),-})}_{\mathrm{H}^a_-}.
\end{equation*}
Thus, we have the following decomposition of the representation variety: $\mathcal{R}_\iota(\overline{Q},V) = \mathcal{R}(\overline{Q},V_+) \oplus \mathcal{R}(\overline{Q},V_-)$. As such, we can identify the graded algebra $\widetilde{\mathcal{O}}$ with the graded algebra $\mathcal{O}_+ \otimes \mathcal{O}_-$:
\begin{equation*}
    \widetilde{\mathcal{O}} = \mathrm{Sym}(\mathcal{R}_\iota(\overline{Q},V)^*) \cong \mathrm{Sym}(\mathcal{R}(\overline{Q},V_+)^*) \otimes \mathrm{Sym}(\mathcal{R}(\overline{Q},V_-)^*) = \mathcal{O_+} \otimes \mathcal{O_-}.
\end{equation*}
Given this identification, the BV operators are related by 
\begin{equation*}
    \widetilde{\Delta} = \widetilde{\Delta}_+ \otimes \mathrm{id}_{\mathcal{O}_-} + \mathrm{id}_{\mathcal{O}_+} \otimes \widetilde{\Delta}_-.
\end{equation*}
To see this, recall that the BV operators arise as inverses to natural supertrace pairings, as given in \eqref{eqn: pairing}, and that the supertrace is additive over direct sums. More concretely, we can see this from the explicit formula \eqref{eqn: BV operator coords}, where we understand that the generator $\tensor{(X^a)}{_{\mu}^{\nu}}$ vanishes if $\mu$ labels a basis vector of $V_{s(a),+}$ and $\nu$ labels a basis vector of $V_{t(a),-}$, for example.

Similarly, the graded algebra morphisms $\varphi$ and $\varphi_{\pm}$ are related by 
\begin{equation*}
    \varphi(x) = \varphi_+(x) \otimes 1_{\mathcal{O}_-} + 1_{\mathcal{O}_+} \otimes \varphi_-(x), \quad \text{for all $x \in \mathscr{A}$}.
\end{equation*}
It follows that for all $x,y \in \mathscr{A}$ (so that $xy \in \mathrm{Sym}^2(\mathscr{A})$) we have
\begin{equation}\label{eqn: relation of trace maps}
    \varphi(xy) = \varphi_+(xy) \otimes 1_{\mathcal{O}_-} + 1_{\mathcal{O}_+} \otimes \varphi_-(xy) + \varphi_+(x) \otimes \varphi_-(y) + (-1)^{xy}\varphi_+(y) \otimes \varphi_-(x).
\end{equation}
Now recall, by Lemma \ref{lem:IBLSymHOm} and Proposition \ref{prop: gerstenhaber}, that $\varphi$ is a BV morphism if and only if 
\begin{equation}\label{eqn:cobracket_intertwining_weakening}
    \varphi(\hbar\widetilde{\delta}(x)) = \widetilde{\Delta}(\varphi(x)), \quad \text{for all $x \in \mathscr{A}$.}
\end{equation}
Writing ${\delta}(x) = x_{(1)} \otimes x_{(2)}$ in Sweedler's notation, so that $\widetilde{\delta}(x) = x_{(1)}x_{(2)}$, the left hand side of \eqref{eqn:cobracket_intertwining_weakening} expands via \eqref{eqn: relation of trace maps} to
\begin{equation*}
    \varphi_+(\hbar\widetilde{\delta}(x))\otimes 1_{\mathcal{O}_-} + 1_{\mathcal{O}_+} \otimes \varphi_-(\hbar\widetilde{\delta}(x)) + 
    \hbar\left(\varphi_+(x_{(1)})\otimes \varphi_-(x_{(2)}) + (-1)^{|x_{(1)}||x_{(2)}|}\varphi_+(x_{(2)})\otimes \varphi_-(x_{(1)})\right).
\end{equation*}
But $\varphi_\pm$ are BV morphisms, so this is nothing but
\begin{equation*}
    \widetilde{\Delta}_+(\varphi_+(x)) \otimes 1_{\mathcal{O}_-} + 1_{\mathcal{O}_+} \otimes \widetilde{\Delta}_-(\varphi_-(x)) + 2\hbar(\varphi_+ \otimes \varphi_-)(\delta(x)),
\end{equation*}
where we have also used symmetry of $\delta$. Meanwhile, the right hand side of \eqref{eqn:cobracket_intertwining_weakening} can be expanded to
\begin{equation*}
    \left(\widetilde{\Delta}_+ \otimes \mathrm{id}_{\mathcal{O}_-} + \mathrm{id}_{\mathcal{O}_+} \otimes \widetilde{\Delta}_-\right)\left(\varphi_+(x) \otimes 1_{\mathcal{O}_-} + 1_{\mathcal{O}_+} \otimes \varphi_-(x)\right)
    = \widetilde{\Delta}_+(\varphi_+(x)) \otimes 1_{\mathcal{O}_-} + 1_{\mathcal{O}_+} \otimes \widetilde{\Delta}_-(\varphi_-(x)).
\end{equation*}
Thus, \eqref{eqn:cobracket_intertwining_weakening} is equivalent to the condition
\begin{equation*}
    (\varphi_+ \otimes \varphi_-)(\delta(x)) = 0, \quad \text{for all $x \in \mathscr{A}$.}
\end{equation*}
Hence we have the following:
\begin{thm} \label{thm:BVrelaxedcondition}
    Given that $p=1 \mod 2$, suppose that the even maps $\{\iota^i\}_{i \in \overline{Q}_0}$ satisfy $(\iota^i)^2 = \hbar^{-1}\mathrm{id}_{V_i}$ for each $i$. Decomposing each $V_i$ into eigenspaces $V_{i,\pm}$ of $\iota^i$, the map $\varphi:\mathscr{B} \to \widetilde{\mathcal{O}}$ is a BV algebra morphism if and only if $\varphi_+ \otimes \varphi_- \circ \delta = 0$, where $\varphi_{\pm}: \mathscr{B} \to \mathcal{O}_\pm$ are the BV morphisms associated to this decomposition.
\end{thm}
In the next section, we will see a basic example of where $\iota$ is not proportional to the identity (i.e. the assumptions of Theorem \ref{thm: BVmorph} are not satisfied).
\section{Examples}\label{section: examples}
We will now explore the necklace Lie bialgebras and associated BV morphisms $\varphi$ for some simple quivers. As a warm-up, let us look at a quiver with one vertex $*$ and no edges. The necklace Lie bialgebra is one-dimensional, spanned by $e_*$, and the bracket and cobracket are zero. The representation variety is just $\Sym 0 = \mathbb C$, and $\varphi$ sends $e_*$ to $\str{\iota_*}$. This intertwines the BV operators, since they both vanish identically, and there are no additional conditions on $\iota_*$. 
\smallskip

Next, let us look at the $1$-edge quiver
\begin{equation} \label{eq:oneedge} \begin{tikzcd}
	1 & 2
	\arrow["a", from=1-1, to=1-2]
\end{tikzcd}\end{equation}
The double quiver has an additional arrow $\bar{a}$. Note that a non-constant cyclic path can be written as $(\bar{a} a)^k$, so let us introduce the notation $x := \bar{a} a$. The necklace Lie algebra depends on the parameter $p$ as follows.
\begin{itemize}
    \item[\fbox{$p=0$:}] In this case, the necklace Lie bialgebra is spanned by cyclic paths $e_1$, $e_2$ and $x^k = (\bar{a} a)^k = (a \bar{a})^k$ for $k\ge 1$. The bracket vanishes identically, while the cobracket satisfies
    \[ \delta (x) = - e_1 \wedge e_2 \; \text{ and } \; \delta (x^k) = -k( e_1\wedge x^{k-1} + x^{k-1}\wedge e_2) \text{ for } k>1.\]
    Here, the normalization of the wedge product is given by $e_1\wedge e_2 = \frac 12 (e_1\otimes e_2 - e_2 \otimes e_1)$. By Theorem \ref{thm: BVmorph}, any $\mathbb{Z}_2$-graded vector spaces with properly normalized $\iota^2$ will give rise to a morphism of BV algebras $\varphi$.
    
    \item[\fbox{$p=1$:}] The necklace Lie bialgebra is spanned by $e_1$, $e_2$ and $x^k = (\bar{a} a)^k = (a \bar{a})^k$ for $k\ge 1$ \emph{odd}; the elements $x^k$ are odd. For even $k$, $x^k$ vanishes due to the sign appearing in the cyclic relation \eqref{eq:cyclicpropexplicit}. Once again, the bracket vanishes, while 
    \[\delta (x) = \frac 12 (e_1\otimes e_2 + e_2 \otimes e_1) \; \text{ and } \; \delta (x^k) = 0 \text{ for } k > 1.\]   
    The condition for $\varphi$ to be a morphism of BV algebras will, in this case, depend on the decoration of the quiver in an interesting way. If we decorate vertex $1$ with $(V, \iota)$ and vertex $2$ with $(W, \kappa)$ such that $\iota^2 = \kappa^2 = \hbar^{-1}$, then the condition $\varphi_+ \otimes \varphi_- \circ \delta = 0$ from Theorem \ref{thm:BVrelaxedcondition} is non-trivial only on the cyclic path $x$, where it gives
    \[ 0 = \varphi_+(e_1)\otimes \varphi_-(e_2) + \varphi_+(e_2)\otimes\varphi_-(e_1) = \left(\str(\iota_+)\str(\kappa_-) + \str(\kappa_+) \str(\iota_-)\right) 1_{{\mathcal{O}}_+} \otimes 1_{{\mathcal{O}_-}},\]
     or, equivalently,
     \begin{equation}\label{eq:dimcond} 0 = \sdim V_+ \sdim W_- + \sdim W_+ \sdim V_- ,\end{equation}
     where $V_{\pm}$ and $W_\pm$ are the positive and negative eigenspaces of $\iota$ and $\kappa$. This condition has other interesting solutions than the ones satisfying the hypotheses of Theorem \ref{thm: BVmorph} --- those correspond to either the positive eigenspaces vanishing, or the negative eigenspaces vanishing, for both $V$ and $W$.
\end{itemize}

    Let us finish by two remarks on more general quivers. The one-vertex, one-loop quiver has a more complicated necklace Lie bialgebra: it is spanned by cyclic words in two non-commuting letters. The representation variety is, for $p=0$, given by $q(n)^{\oplus 2}$, with $n = \dim V_\text{even} = \dim V_\text{odd}$. Here, $q(n)$ is the queer Lie superalgebra, which is isomorphic to endomorphisms of $V$ commuting with $\iota$. For $p=1$, the representation variety is given by $(gl(V_+) \oplus gl(V_-))^{\oplus 2}$.

Finally, for an arbitrary quiver in the $p=1$ case, we expect that the conditions imposed on $\iota$ by Theorem \ref{thm:BVrelaxedcondition} become stronger, so that the only possible $\iota$ for which $\varphi$ is a BV morphism are those proportional to the identity, if the quiver is complicated enough. Note that we get a condition of the form \eqref{eq:dimcond} for every edge of this arbitrary quiver, which follows from the embedding of the one-edge quiver \eqref{eq:oneedge}, inducing a commuting square of BV morphisms.

\appendix
\section{Possibility of \texorpdfstring{$\mathbb Z$}{Z}-grading}\label{Appendix:ZZ2}
Let us now sketch a possible extension of this work to a $\mathbb Z$-graded setting. Our necklace Lie bialgebra is naturally $\mathbb Z$-graded, with new edges having degree $p$. The necklace Lie bracket and cobracket have degree $-p$. On the shifted space $\mathscr{A}[p+1]$, the shifted bracket has degree $+1$, while the shifted cobracket has degree $-2p-1$. 

To get a Batalin-Vilkovisky algebra structure, one should add $\hbar$ as a formal variable, i.e. consider $(\Sym \mathscr{A}[p+1])[\hbar]$ with $\hbar$ having degree $2p+2$. The BV operator will be given by $\Delta = \widetilde{\operatorname{br}} + \hbar \widetilde{\delta}$.

On the representation variety side, the situation is less clear. It still makes sense to ask for each vertex $i$ of the quiver to be equipped with a vector space $V_i$, an invertible linear map $\iota$ of degree $-p-1$, and an invertible linear map $\hbar$ of degree $2p+2$. However, this forces the spaces $V_i$ to be infinite-dimensional (or trivial), as the dimensions of the graded components of $V_i$ will be $p+1$-periodic. Therefore, the pairing \eqref{eqn: pairing} and the BV map $\varphi$ are not defined. One could define a BV operator on functions on the $\iota$-intertwining representation variety directly as in Proposition \ref{prop:BVRepexp}. Then it remains to renormalize the map $\varphi$, by using e.g. only the first $p+1$ graded components in computing the trace. We leave these questions open for future investigation.

\section{Comparison of notions of shifted Lie bialgebras}\label{Appendix:IBL}
Lie bialgebras with arbitrary degrees of the bracket and the cobracket appeared in \cite[Sec.~2.1]{MWLieBialgProp} and \cite[Sec.~1.1]{KMW2016}. Based on the parity of $n = |\operatorname{br}| + |\delta|$, the sum of degrees of the bracket and the cobracket (which is invariant w.r.t. shifting), such Lie bialgebras are called even or odd. 

Necklace Lie bialgebras we construct in our paper are always even, and thus we can shift them to match the grading convention of \cite{MWLieBialgProp}. Concretely, we have.
\begin{prop}
    Let $(\mathscr A, \delta, \operatorname{br})$ be a degree $p$ Lie bialgebra in the sense of our Definition \ref{defn: GLBv2}. Then $\mathscr{A}[-p]$ with $\delta^{[-p]}$ and $\operatorname{br}^{[-p]}$ is an involutive Lie $(-2p)$-bialgebra in the sense of \cite[Sec.~2.1]{MWLieBialgProp}.
\end{prop}
Note that this shift by $-p$ is the other possible case $\mod \mathbb Z_2$, compared to the shift by $p+1$ we performed in Proposition \ref{prop: shiftIBL}.
\begin{proof}
    After the shift by $[-p]$, the cobracket is of degree $0$, while the bracket is of degree $[-2p]$.  For the shifted bracket and cobracket, we need check that they satisfy the requirements from \cite[Sec.~2.1]{MWLieBialgProp}. For $p=0 \mod 2$, we are in the case of usual Lie bialgebras in both cases, and no shift is necessary. For $p=1 \mod 2$, we use Equation \eqref{eq:shiftsfortensors} and the sign coming from a flip. The symmetry properties, Jacobi, coJacobi and involutivity follow easily. The cocycle condition follows after some manipulations, let us just say that terms 2 and 3 in \eqref{eq:pictorialcocycle} acquire a minus sign, and match with terms 2 and 4 of the RHS of the cocycle condition in \cite[Sec.~2.1]{MWLieBialgProp}.
\end{proof}

The connection with $\mathsf{IBL}_\infty$ algebras of \cite{CFL} is more direct. By Proposition \ref{prop: shiftIBL}, a degree $p$ Lie bialgebra $(\mathscr{A}, \operatorname{br}, \delta)$ induces a BV operator on $\Sym(\mathscr A[p+1])[\hbar]$ given by $b + \hbar c$, where $b = \operatorname{br}^{[p+1]}$, $c = \delta^{[p+1]}$ and $\hbar$ has degree $2p+2$.

On the other hand, an $\mathsf{IBL}_\infty$ algebra  with parameter $d\in \mathbb Z$ on a $\mathbb Z$-graded vector space\footnote{This $D$ is $C[1]$ in the notation of \cite{CFL}, note that for the definition of an $\mathsf{IBL}_\infty$, only the grading on $D$ plays a role.} $D$ is given by a sequence of operations $\mathfrak{p}_{k,l,g}\colon \Sym^k D \to \Sym^l D$ of degree $|\mathfrak p_{k, l, g}| = -2(k+g-1) -1$, such that their sum squares to zero \cite[Def.~2.3]{CFL}.

To get an $\mathsf{IBL}_\infty$ algebra from $(\mathscr{A}, \operatorname{br}, \delta)$, we just need to match the two grading conventions\footnote{In \cite{CFL}, auxiliary formal variables $\hbar$ and $\tau$ are used to separate different terms of the defining $\mathrm{IBL}_\infty$ identity $\Delta^2=0$. In our case, with only two non-zero operations, $\hbar$ and $\tau$ don't play any role.}. For example, we can set $d=-1$ and $D = \mathscr A[p+1]$ with two non-trivial operations $\mathfrak p_{2, 1, 0} := b$, $\mathfrak p_{1, 2, -p} := c$. Other solutions are also possible, by reversing or further shifting the grading of $\mathscr A$. For example, on $\mathscr{A}[p+1+2a]$, we can set $p_{2, 1, g_b} := b^{[2a]}$, $p_{1, 2, g_c} := c^{[2a]}$, for any $a, d\in \mathbb Z$ and $g_b, g_c \ge 0$ satisfying $-d(1+g_b) = 1 + a$ and $dg_c = p+a$.

\section{Proof of Proposition \ref{prop:GSNLB} (\texorpdfstring{$p=0$}{p=0} case)}\label{Appendix:Cals}
In this calculation we will intentionally ignore the edge cases relating to the appearances of constant paths, assuring the reader that everything works out if they are treated properly.

To aid in our calculations, we introduce new summation notations. To see how these work, suppose we have a closed path $a_1 \cdots a_n$ which we denote by $A$ when forgetting endpoints. Then, for example,
\begin{itemize}
    \item $\; \mathrel{\raisebox{-1.8ex}{\includegraphics[scale=1]{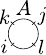}}}\; $ is the instruction to sum over all tuples of indices $(i,j,k,l)$ labelling arrows in $A$; and

    \item $\; \mathrel{\raisebox{-1.8ex}{\includegraphics[scale=1]{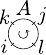}}}\;$ is the instruction to sum over only those tuples $(i,j,k,l)$ where the indices are pairwise distinct and subject to the depicted cyclic ordering.
\end{itemize}

Now fix $p=0$. The bracket and cobracket can be written as
\begin{equation*}
    \operatorname{br}(A,B) = \; 
    \mathrel{\raisebox{-3.1ex}{\includegraphics[scale=1]{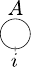}}}\;
    \mathrel{\raisebox{-3.5ex}{\includegraphics[scale=1]{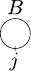}}}
    \langle a_i, b_j\rangle A_{ii}B_{jj}
    \quad \text{and} \quad
    \delta(A) = \;\mathrel{\raisebox{-1.8ex}{\includegraphics[scale=1]{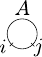}}} \langle a_i, a_j \rangle A_{ij} \otimes A_{ji},
\end{equation*}
where we are safe to drop the indices $k$ and $l$ appearing in \eqref{eqn: Bracket defn}, for example, since the endpoints of closed paths in $\mathscr{A}$ are immaterial in the $p=0$ case. 

It suffices to show that the identities \eqref{gb1}--\eqref{gb6} of Definition \ref{defn: GLBv2} are satisfied on the spanning sets of closed paths in $\overline{Q}$ with enpoints forgotten. In what follows, $A,B$ and $C$ denote such paths.
\begin{enumerate}[(i)]
    \item  ``$\operatorname{br}(x,y) = -\operatorname{br}(y,x)$'' --- \textit{Bracket symmetry}:
    \begin{equation*}
        \operatorname{br}(A,B)
        =
        \; 
        \mathrel{\raisebox{-3.1ex}{\includegraphics[scale=1]{CustomSummationNotation/A-i-0.pdf}}}\;
        \mathrel{\raisebox{-3.5ex}{\includegraphics[scale=1]{CustomSummationNotation/B-j-0.pdf}}}
        \underbrace{\langle a_i, b_j\rangle}_{-\langle b_j, a_i \rangle} \underbrace{A_{ii}B_{jj}}_{B_{jj}A_{ii}}
        =
        -
        \;
        \mathrel{\raisebox{-3.5ex}{\includegraphics[scale=1]{CustomSummationNotation/B-j-0.pdf}}}
        \; 
        \mathrel{\raisebox{-3.1ex}{\includegraphics[scale=1]{CustomSummationNotation/A-i-0.pdf}}}
        \langle b_j, a_i \rangle B_{jj}A_{ii}
        = -\operatorname{br}(B,A).
    \end{equation*}

    \item ``$\mathrm{im}(\delta) \subset \langle x \otimes y - y \otimes x \rangle$'' --- \textit{Cobracket symmetry}:
    \begin{align*}
        \delta(A)
        &=
        \;\mathrel{\raisebox{-1.8ex}{\includegraphics[scale=1]{CustomSummationNotation/A-ij-0.pdf}}} \langle a_i, a_j \rangle A_{ij} \otimes A_{ji}
        =
        \left(\;\mathrel{\raisebox{-3.1ex}{\includegraphics[scale=1]{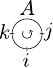}}} + \mathrel{\raisebox{-3.5ex}{\includegraphics[scale=1]{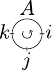}}}\right) \langle a_i, a_j \rangle A_{ij} \otimes A_{ji}\\
        &=
        \;\mathrel{\raisebox{-3.1ex}{\includegraphics[scale=1]{CustomSummationNotation/A-kij-1.pdf}}} \langle a_i, a_j \rangle \left(A_{ij} \otimes A_{ji} - A_{ji} \otimes A_{ij} \right),
    \end{align*}
    where $k$ is any fixed index, and here (and \textit{only} here) we do not require $k$ to be distinct from $i$ and $j$.

    \item  ``$\operatorname{br}(\operatorname{br} \otimes 1) \mathfrak{S}_3 = 0$'' --- \textit{Jacobi}:
    \begin{align*}
        \operatorname{br}(\operatorname{br}(A,B),C) 
        &=
        \; \mathrel{\raisebox{-3.1ex}{\includegraphics[scale=1]{CustomSummationNotation/A-i-0.pdf}}}\;
        \mathrel{\raisebox{-3.5ex}{\includegraphics[scale=1]{CustomSummationNotation/B-j-0.pdf}}}
        \langle a_i,b_j \rangle
        \operatorname{br}(A_{ii}B_{jj},C)\\
        &=
        \; \mathrel{\raisebox{-3.1ex}{\includegraphics[scale=1]{CustomSummationNotation/A-i-0.pdf}}}\;
        \mathrel{\raisebox{-3.5ex}{\includegraphics[scale=1]{CustomSummationNotation/B-j-0.pdf}}}
        \langle a_i,b_j \rangle
        \left(
        \; \mathrel{\raisebox{-3.1ex}{\includegraphics[scale=1]{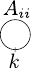}}}
        \; \mathrel{\raisebox{-3.1ex}{\includegraphics[scale=1]{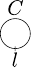}}}
        \langle a_k, c_l\rangle \underbrace{(A_{ii}B_{jj})_{kk}}_{A_{ki}B_{jj}A_{ik}}C_{ll}\; 
        +
        \; \mathrel{\raisebox{-3.1ex}{\includegraphics[scale=1]{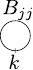}}}
        \; \mathrel{\raisebox{-3.1ex}{\includegraphics[scale=1]{CustomSummationNotation/C-l-0.pdf}}}
        \langle b_k, c_l\rangle \underbrace{(A_{ii}B_{jj})_{kk}}_{B_{kj}A_{ii}B_{jk}}C_{ll}
        \right)\\
        &=
        \underbrace{
            \; 
            \mathrel{\raisebox{-3.3ex}{\includegraphics[scale=1]{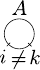}}}
            \mathrel{\raisebox{-3.5ex}{\includegraphics[scale=1]{CustomSummationNotation/B-j-0.pdf}}}
            \; 
            \mathrel{\raisebox{-3.1ex}{\includegraphics[scale=1]{CustomSummationNotation/C-l-0.pdf}}}
            \langle a_i, b_j \rangle \langle a_k, c_l \rangle A_{ki} B_{jj} A_{ik} C_{ll}}_{=: X_{ABC}} \;
        +
        \underbrace{
            \; 
            \mathrel{\raisebox{-3.1ex}{\includegraphics[scale=1]{CustomSummationNotation/A-i-0.pdf}}}
            \mathrel{\raisebox{-3.3ex}{\includegraphics[scale=1]{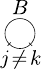}}}
            \mathrel{\raisebox{-3.1ex}{\includegraphics[scale=1]{CustomSummationNotation/C-l-0.pdf}}}
            \langle a_i, b_j \rangle \langle b_k, c_l \rangle B_{kj} A_{ii} B_{jk} C_{ll}}_{=: Y_{ABC}}.
    \end{align*}
    Now notice that we have
    \begin{align*}
        X_{BCA} &= \; 
            \mathrel{\raisebox{-3.3ex}{\includegraphics[scale=1]{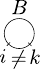}}}
            \mathrel{\raisebox{-3.5ex}{\includegraphics[scale=1]{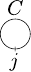}}}
            \; 
            \mathrel{\raisebox{-3.1ex}{\includegraphics[scale=1]{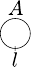}}}
            \langle b_i, c_j \rangle \langle b_k, a_l \rangle B_{ki} C_{jj} B_{ik} A_{ll}\\
            &= \;
            \mathrel{\raisebox{-3.3ex}{\includegraphics[scale=1]{CustomSummationNotation/B-jk-2.pdf}}}
            \mathrel{\raisebox{-3.1ex}{\includegraphics[scale=1]{CustomSummationNotation/C-l-0.pdf}}}
            \;
            \mathrel{\raisebox{-3.1ex}{\includegraphics[scale=1]{CustomSummationNotation/A-i-0.pdf}}}
            \langle b_k,c_l\rangle \underbrace{\langle b_j, a_i \rangle}_{-\langle a_i,b_j \rangle} \underbrace{B_{jk}C_{ll}B_{kj}A_{ii}}_{B_{kj}A_{ii}B_{jk}C_{ll}}
            \quad \text{(re-indexing with the permutation $(ikjl)$)}\\
            &= -Y_{ABC},
    \end{align*}
    and so
    \begin{equation*}
        \operatorname{br}(\operatorname{br}(A,B),C) + \operatorname{br}(\operatorname{br}(B,C),A) + \operatorname{br}(\operatorname{br}(C,A),B)
        = 0.
    \end{equation*}

    \item ``$\mathfrak{S}_3(\delta \otimes 1) \delta = 0$'' --- \textit{Co-Jacobi}:
    \begin{align*}
        (\delta \otimes 1) \delta (A) 
        &=
        (\delta \otimes 1)
        \left(
            \;\mathrel{\raisebox{-1.8ex}{\includegraphics[scale=1]{CustomSummationNotation/A-ij-0.pdf}}} \langle a_i, a_j \rangle A_{ij} \otimes A_{ji}
        \right)
        =
        \;\mathrel{\raisebox{-1.8ex}{\includegraphics[scale=1]{CustomSummationNotation/A-ij-0.pdf}}} \langle a_i, a_j \rangle \delta(A_{ij}) \otimes A_{ji}\\
        &=
        \;\mathrel{\raisebox{-1.8ex}{\includegraphics[scale=1]{CustomSummationNotation/A-ij-0.pdf}}}
        \;\mathrel{\raisebox{-1.6ex}{\includegraphics[scale=1]{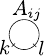}}} 
        \langle a_i, a_j \rangle \langle a_k, a_l \rangle 
        (A_{ij})_{kl} \otimes (A_{ij})_{lk} \otimes A_{ji}\\
        &=
        \; \mathrel{\raisebox{-1.8ex}{\includegraphics[scale=1]{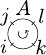}}}
        \langle a_i, a_j \rangle \langle a_k, a_l \rangle 
        \underbrace{(A_{ij})_{kl}}_{A_{kl}} \otimes \underbrace{(A_{ij})_{lk}}_{A_{lj}A_{ik}} \otimes A_{ji}\;
        +
        \mathrel{\raisebox{-1.8ex}{\includegraphics[scale=1]{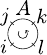}}}
        \langle a_i, a_j \rangle \langle a_k, a_l \rangle 
        \underbrace{(A_{ij})_{kl}}_{A_{kj}A_{il}} \otimes \underbrace{(A_{ij})_{lk}}_{A_{lk}} \otimes A_{ji}.
    \end{align*}
    Re-indexing the second term above with the permutation $(iljk)$, it becomes
    \begin{equation*}
        \mathrel{\raisebox{-1.8ex}{\includegraphics[scale=1]{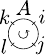}}}
        \underbrace{\langle a_l,a_k \rangle}_{-\langle a_k,a_l \rangle} \langle a_i, a_j \rangle
        \underbrace{A_{ik}A_{lj}}_{A_{lj}A_{ik}} \otimes A_{ji} \otimes A_{kl},
    \end{equation*}
    so $(\delta \otimes 1)\delta(A)$ is given by
    \begin{equation*}
        \mathrel{\raisebox{-1.8ex}{\includegraphics[scale=1]{CustomSummationNotation/A-iklj-1.pdf}}}
        \langle a_i,a_j \rangle \langle a_k,a_l \rangle
        \left(
        A_{kl}\otimes A_{lj}A_{ik} \otimes A_{ji} - A_{lj}A_{ik} \otimes A_{ji} \otimes A_{kl}
        \right),
    \end{equation*}
    from which we see that $\mathfrak{S}_3(\delta \otimes 1) \delta (A) = 0.$
    
    \item ``$\operatorname{br}\circ \delta = 0$'' --- \textit{Involutivity}:
    \begin{align*}
        \operatorname{br}(\delta(A))
        &= 
        \operatorname{br}\left(
        \;\mathrel{\raisebox{-1.8ex}{\includegraphics[scale=1]{CustomSummationNotation/A-ij-0.pdf}}} \langle a_i, a_j \rangle A_{ij} \otimes A_{ji}
        \right)
        =
        \;\mathrel{\raisebox{-1.8ex}{\includegraphics[scale=1]{CustomSummationNotation/A-ij-0.pdf}}}
        \;\mathrel{\raisebox{-3ex}{\includegraphics[scale=1]{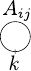}}}
        \;\mathrel{\raisebox{-3ex}{\includegraphics[scale=1]{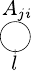}}}
        \langle a_i,a_j \rangle \langle a_k,a_l \rangle \underbrace{(A_{ij})_{kk}}_{A_{kj}A_{ik}} \underbrace{(A_{ji})_{ll}}_{A_{li}A_{jl}}\\
        &=
        \; \mathrel{\raisebox{-1.8ex}{\includegraphics[scale=1]{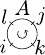}}}
        \langle a_i,a_j \rangle \langle a_k,a_l \rangle
        A_{kj}A_{ik}A_{li}A_{jl}\\
        &=
        \; \mathrel{\raisebox{-1.8ex}{\includegraphics[scale=1]{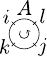}}}
        \langle a_k,a_l \rangle \underbrace{\langle a_j,a_i \rangle}_{-\langle a_i,a_j \rangle}
        A_{jl}A_{kj}A_{ik}A_{li} \quad \text{(re-indexing with the permutation $(ikjl)$)}\\
        &= - \operatorname{br}(\delta(A)),
    \end{align*}
    so $\operatorname{br}(\delta(A)) = 0$.

    \item ``$\delta(\operatorname{br}(x,y)) = \mathrm{ad}_x(\delta(y)) - \mathrm{ad}_y(\delta(x))$'' --- \textit{Cocycle condition}:
    \begin{align*}
        \delta(\operatorname{br}(A,B))
        =
        \delta\left(
        \; 
        \mathrel{\raisebox{-3.1ex}{\includegraphics[scale=1]{CustomSummationNotation/A-i-0.pdf}}}\;
        \mathrel{\raisebox{-3.5ex}{\includegraphics[scale=1]{CustomSummationNotation/B-j-0.pdf}}}
        \langle a_i, b_j\rangle \underbrace{A_{ii}B_{jj}}_{=: C}
        \right)
        =
        \; 
        \mathrel{\raisebox{-3.1ex}{\includegraphics[scale=1]{CustomSummationNotation/A-i-0.pdf}}}\;
        \mathrel{\raisebox{-3.5ex}{\includegraphics[scale=1]{CustomSummationNotation/B-j-0.pdf}}}
        \mathrel{\raisebox{-1.7ex}{\includegraphics[scale=1]{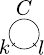}}}
        \langle a_i, b_j\rangle \langle c_k, c_l \rangle C_{kl} \otimes C_{lk}.
    \end{align*}
    Leaving implicit that $k \neq i,j$ and $l \neq i,j$ (where applicable), as a summation instruction we have
    \begin{equation*}
        \mathrel{\raisebox{-3.1ex}{\includegraphics[scale=1]{CustomSummationNotation/A-i-0.pdf}}}\;
        \mathrel{\raisebox{-3.5ex}{\includegraphics[scale=1]{CustomSummationNotation/B-j-0.pdf}}}
        \mathrel{\raisebox{-1.7ex}{\includegraphics[scale=1]{CustomSummationNotation/C-kl-0.pdf}}}\;
        =
        \underbrace{
        \;
        \mathrel{\raisebox{-3.1ex}{\includegraphics[scale=1]{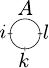}}}\;
        \mathrel{\raisebox{-3.5ex}{\includegraphics[scale=1]{CustomSummationNotation/B-j-0.pdf}}}\;}_{\rightsquigarrow \; \circled{1}}
        +
        \underbrace{
        \;
        \mathrel{\raisebox{-3.1ex}{\includegraphics[scale=1]{CustomSummationNotation/A-i-0.pdf}}}\;
        \mathrel{\raisebox{-3.1ex}{\includegraphics[scale=1]{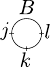}}}\;}_{\rightsquigarrow \; \circled{2}}
        +
        \underbrace{
        \;
        \mathrel{\raisebox{-1.8ex}{\includegraphics[scale=1]{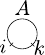}}}\;
        \mathrel{\raisebox{-1.8ex}{\includegraphics[scale=1]{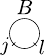}}}\;}_{\rightsquigarrow \; \circled{3}}
        +
        \underbrace{
        \;
        \mathrel{\raisebox{-1.8ex}{\includegraphics[scale=1]{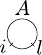}}}\;
        \mathrel{\raisebox{-1.8ex}{\includegraphics[scale=1]{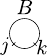}}}\;}_{\rightsquigarrow \; \circled{4}},
    \end{equation*}
    where we have introduced $\circled{1}, \dots, \circled{4}$ to denote the contributions to $\delta(\operatorname{br}(A,B))$ due to the relevant portions of the summation. Now,
    \begin{align*}
        \circled{1} &=
        \;
        \mathrel{\raisebox{-3.1ex}{\includegraphics[scale=1]{CustomSummationNotation/A-ikl-0.pdf}}}\;
        \mathrel{\raisebox{-3.5ex}{\includegraphics[scale=1]{CustomSummationNotation/B-j-0.pdf}}}
        \langle a_i, b_j\rangle \langle a_k, a_l \rangle (A_{ii}B_{jj})_{kl} \otimes (A_{ii}B_{jj})_{lk}\\
        &=\;
        \mathrel{\raisebox{-3.1ex}{\includegraphics[scale=1]{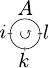}}}\;
        \mathrel{\raisebox{-3.5ex}{\includegraphics[scale=1]{CustomSummationNotation/B-j-0.pdf}}}
        \langle a_i, b_j\rangle \langle a_k, a_l \rangle 
        A_{kl} \otimes \underbrace{A_{li}B_{jj}A_{ik}}_{(A_{lk})_{ii}B_{jj}}\;
        +\;
        \mathrel{\raisebox{-3.1ex}{\includegraphics[scale=1]{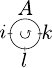}}}\;
        \mathrel{\raisebox{-3.5ex}{\includegraphics[scale=1]{CustomSummationNotation/B-j-0.pdf}}}
        \langle a_i, b_j\rangle \langle a_k, a_l \rangle \underbrace{A_{ki}B_{jj}A_{il}}_{(A_{kl})_{ii}B_{jj}} \otimes A_{lk}.
    \end{align*}
    As summation instructions, we have
    \begin{equation*}
        \mathrel{\raisebox{-3.1ex}{\includegraphics[scale=1]{CustomSummationNotation/A-ikl-1.pdf}}}
        \;
        =
        \;
        \mathrel{\raisebox{-1.8ex}{\includegraphics[scale=1]{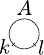}}}\;
        \mathrel{\raisebox{-3.1ex}{\includegraphics[scale=1]{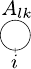}}} \quad \text{and} \quad
        \mathrel{\raisebox{-3.1ex}{\includegraphics[scale=1]{CustomSummationNotation/A-ilk-1.pdf}}}
        \;
        =
        \;
        \mathrel{\raisebox{-1.8ex}{\includegraphics[scale=1]{CustomSummationNotation/A-kl-0.pdf}}}\;
        \mathrel{\raisebox{-3.1ex}{\includegraphics[scale=1]{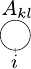}}}
    \end{equation*}
    (more correctly, these should have $k \neq l$ --- but the $k=l$ terms are zero anyway due to the appearance of the factor $\langle a_k, a_l \rangle$), so
    \begin{align*}
        \circled{1}
        &=\;
        \mathrel{\raisebox{-1.8ex}{\includegraphics[scale=1]{CustomSummationNotation/A-kl-0.pdf}}}
        \langle a_k, a_l \rangle
        A_{kl}
        \otimes 
        \underbrace{\left( 
        \mathrel{\raisebox{-3.1ex}{\includegraphics[scale=1]{CustomSummationNotation/Alk-i-0.pdf}}}\;
        \mathrel{\raisebox{-3.5ex}{\includegraphics[scale=1]{CustomSummationNotation/B-j-0.pdf}}}\;
        \langle a_i, b_j \rangle
        (A_{lk})_{ii}B_{jj}
        \right)}_{\operatorname{br}(A_{lk},B) = -\mathrm{ad}_B(A_{lk})}\\
        & \quad\quad\quad\quad\;+\;
        \mathrel{\raisebox{-1.8ex}{\includegraphics[scale=1]{CustomSummationNotation/A-kl-0.pdf}}}
        \langle a_k, a_l \rangle
        \underbrace{
        \left( 
        \mathrel{\raisebox{-3.1ex}{\includegraphics[scale=1]{CustomSummationNotation/Akl-i-0.pdf}}}\;
        \mathrel{\raisebox{-3.5ex}{\includegraphics[scale=1]{CustomSummationNotation/B-j-0.pdf}}}\;
        \langle a_i, b_j \rangle
        (A_{kl})_{ii}B_{jj}
        \right)}_{\operatorname{br}(A_{kl},B) = -\mathrm{ad}_B(A_{kl})}
        \otimes A_{lk}\\
        &=
        -(1 \otimes \mathrm{ad}_B + \mathrm{ad}_B \otimes 1)
        \left(
        \;\mathrel{\raisebox{-1.8ex}{\includegraphics[scale=1]{CustomSummationNotation/A-kl-0.pdf}}} \langle a_k, a_l \rangle A_{kl} \otimes A_{lk}
        \right)\\
        &=
        -\mathrm{ad}_B(\delta(A)).
    \end{align*}
    Similarly, $\circled{2} = \mathrm{ad}_A(\delta(B))$. Meanwhile,
    \begin{align*}
        \circled{4}
        &=\;
        \mathrel{\raisebox{-1.8ex}{\includegraphics[scale=1]{CustomSummationNotation/A-il-0.pdf}}}\;
        \mathrel{\raisebox{-1.8ex}{\includegraphics[scale=1]{CustomSummationNotation/B-jk-0.pdf}}}\;
        \langle a_i, b_j\rangle \langle b_k, a_l \rangle \underbrace{(A_{ii}B_{jj})_{kl}}_{B_{kj}A_{il}} \otimes \underbrace{(A_{ii}B_{jj})_{lk}}_{A_{li}B_{jk}}\\
        &=\;
        \mathrel{\raisebox{-1.8ex}{\includegraphics[scale=1]{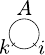}}}\;
        \mathrel{\raisebox{-1.8ex}{\includegraphics[scale=1]{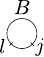}}}\;
        \langle a_k, b_l\rangle \underbrace{\langle b_j, a_i \rangle}_{-\langle a_i,b_j \rangle} \underbrace{B_{jl}A_{ki}}_{(A_{ii}B_{jj})_{kl}} \otimes \underbrace{A_{ik}B_{lj}}_{(A_{ii}B_{jj})_{lk}} \quad \text{(re-indexing with the permutation $(ikjl)$)}\\
        &= - \circled{3}.
    \end{align*}
    So, 
    \begin{equation*}
        \delta(\operatorname{br}(A,B)) = \mathrm{ad}_A(\delta{(B)}) - \mathrm{ad}_B(\delta{(A)}).
    \end{equation*}
\end{enumerate}
This completes the calculation.

\printbibliography
\end{document}